\documentclass[a4paper,reqno]{amsart}

\usepackage[french,british]{babel}
\usepackage{amssymb,amsmath,amsthm,amscd,mathrsfs,graphicx,color,multicol,booktabs,bm,manfnt}
\usepackage[cmtip,all,matrix,arrow,tips,curve]{xy}

\usepackage{mathabx}

\usepackage{manfnt}
\usepackage{centernot}   

\newcommand{\CC}{{\mathbb C}}
\newcommand{\cA}{{\mathscr A}}

\newcommand{\cB}{{\mathscr B}}
\newcommand{\cC}{{\mathscr C}}

\newcommand{\cE}{{\mathscr E}}
\newcommand{\cF}{{\mathscr F}}
\newcommand{\cG}{{\mathscr G}}
\newcommand{\cH}{{\mathscr H}}
\newcommand{\cI}{{\mathscr I}}
\newcommand{\cJ}{{\mathscr J}}
\newcommand{\cK}{{\mathscr K}}
\newcommand{\cL}{{\mathscr L}}
\newcommand{\cM}{{\mathscr M}}
\newcommand{\cN}{{\mathscr N}}

\newcommand{\cO}{{\mathscr O}}

\newcommand{\cS}{{\mathscr S}}

\newcommand{\cX}{{\mathscr X}}

\newcommand{\dra}{\dashrightarrow}

\newcommand{\Ext}{\text{Ext}}

\newcommand{\hra}{\hookrightarrow}
\newcommand{\KK}{{\mathbb K}}
\newcommand{\la}{\langle}

\newcommand{\lra}{\longrightarrow}

\newcommand{\NN}{{\mathbb N}}
\newcommand{\ov}{\overline}
\newcommand{\PP}{{\mathbb P}}
\newcommand{\QQ}{{\mathbb Q}}
\newcommand{\ra}{\rangle}
\newcommand{\RR}{{\mathbb R}}

\newcommand{\vv}{{\bf v}}

\newcommand{\wt}{\widetilde}
\newcommand{\ww}{{\bf w}}
\newcommand{\ZZ}{{\mathbb Z}}
\newcommand{\Gr}{\mathrm{Gr}}

\theoremstyle{plain}

\newtheorem{thm}{Theorem}[section]
\newtheorem*{thm*}{Theorem}

\newtheorem{clm}[thm]{Claim}

\newtheorem{crl}[thm]{Corollary}

\newtheorem*{hyp*}{Hypothesis}
\newtheorem{lmm}[thm]{Lemma}
\newtheorem{prp}[thm]{Proposition}
\newtheorem{prp-dfn}[thm]{Proposition-Definition}

\theoremstyle{definition}

\newtheorem{ass}[thm]{Assumption}
\newtheorem{dfn}[thm]{Definition}

\theoremstyle{remark}

\newtheorem{expl}[thm]{Example}

\newtheorem*{qst*}{Main Question}
\newtheorem{rmk}[thm]{Remark}

\DeclareMathOperator{\Amp}{Amp}

\DeclareMathOperator{\Aut}{Aut}

\DeclareMathOperator{\ch}{ch}

\DeclareMathOperator{\cl}{cl}

\DeclareMathOperator{\Def}{Def}

\DeclareMathOperator{\disc}{disc}

\DeclareMathOperator{\End}{End}

\DeclareMathOperator{\Hom}{Hom}

\DeclareMathOperator{\Kum}{Kum}

\DeclareMathOperator{\NS}{NS}

\DeclareMathOperator{\Pic}{Pic}

\DeclareMathOperator{\sing}{sing}

\DeclareMathOperator{\Sym}{Sym}
\DeclareMathOperator{\td}{Td}

\DeclareMathOperator{\Tr}{Tr}

\usepackage{multirow}

 \makeindex
 
\begin{document}
 \title{Moduli of sheaves on $K3$'s  and  higher dimensional HK varieties}
 \author{Kieran G. O'Grady}
\dedicatory{Dedicated to Shing-Tung Yau on the occasion of his 70th birthday}
\date{\today}
\thanks{Partially supported by PRIN 2017}
\begin{abstract}
We review a proof of the well know result stating that moduli spaces of stable sheaves with fixed Chern character  on a polarized $K3$ surface are deformations of a hyperk\"ahler variety of Type $K3^{[n]}$ (if a suitable numerical hypothesis is satisfied). In a recent work we have adapted  that proof in order to prove results on moduli of vector bundles on polarized hyperk\"ahler varieties of Type $K3^{[2]}$ - this is the content of the second part of the paper.

\end{abstract}

  \maketitle
\bibliographystyle{amsalpha}
\section{Introduction}\label{sec:intro}
\setcounter{equation}{0}
\subsection{Background and general outline}\label{subsec:retromotivi}
\setcounter{equation}{0}
Vector bundles, or more generally coherent sheaves, on $K3$ surfaces play a prominent r\^ole in Algebraic Geometry. We mention a few instances of interesting results obtained by studying sheaves on $K3$'s.
First moduli spaces of semistable 
sheaves on polarized $K3$ surfaces provide, possibly after desingularization,   models for \lq\lq half\rq\rq\ of the known deformation classes of (compact) hyperk\"ahler manifolds - the other half being provided by moduli spaces of sheaves on abelian surfaces. As shown by Mukai, also zero-dimensional moduli spaces are interesting. In fact  rigid stable vector bundles of rank  greater than $1$ allow to describe explicitly the  ideal sheaf of  embedded $K3$ surfaces of certain genera for which the surface is not a complete intersection.
 Lastly, we recall Lazarsfeld's proof of the Giesker-Petri Theorem via vector bundles on $K3$'s. 

Since $K3$ surfaces are the hyperk\"ahler (HK) surfaces, one is naturally led to investigate moduli of sheaves on polarized HK varieties. Recently we have proposed to focus on torsion-free sheaves on a HK whose discriminant (if $\cF$ is a vector bundle, the discriminant is equal to $-c_2(\cF^{\vee}\otimes \cF)$) satisfies a certain condition, and we named such sheaves modular. By way of example any torsion free  sheaf on a $K3$ surface is modular, and a sheaf on an arbitrary HK manifold $X$ whose discriminant is a multiple of $c_2(X)$ is modular (but this condition is in no way necessary for a sheaf to be modular). In~\cite{ogmodvb} we proved an existence and uniqueness result for slope-stable modular sheaves on HK's of Type $K3^{[2]}$ 
(i.e.~deformations of the Hilbert square of a $K3$ surface) which is analogous to well known results valid for sheaves on $K3$'s.  

In fact we extended to moduli spaces of sheaves on HK's of Type $K3^{[2]}$  a strategy   that was employed more than $20$ years ago in order to prove results on moduli spaces of sheaves (of arbitrary dimension) on $K3$'s. The idea is to specialize the polarized HK to one which carries a Lagrangian fibration (i.e.~an elliptic $K3$ if the dimension is $2$), and then to relate (semi)stability of a sheaf  on a Lagrangian HK to stability of its restriction to a generic Lagrangian fiber.
The  latter holds for a sheaf $\cF$ on an elliptic $K3$ surface, provided the polarization is close to the fiber of the fibration, because of the well-known decomposition of the ample cone $\Amp(S)_{\RR}$ into connected open chambers with the property that stability (or non stability) of $\cF$ is constant for polarizations belonging to the same chamber. The open chambers are the connected components of the complement of the union of walls $\xi^{\bot}\cap \Amp(S)_{\RR}$, where $\xi\in H^{1,1}_{\ZZ}(S)$  are the classes such that $-r(\cF)^2\Delta(\cF)\le 4\xi^2<0$.  A torsion-free sheaf $\cF$ on a HK variety is modular  exactly if the variation of $h$ slope-stability of $\cF$ behaves as in the $2$ dimensional case.

Another remarkable consequence of modularity is the following. Let $X$ be a HK with a Lagrangian fibration, and let $\cF$ be a modular vector bundle on $X$ 
whose restriction  to a generic Lagrangian fiber  is slope-stable. Then the restriction of $\cF$ to a generic fiber is a semi homogeneous vector bundle, and hence it has no infinitesimal deformations fixing the determinant. It follows that the strategy outlined above in dimensions higher than $2$ has a strict   resemblance to that which has been implemented in the case of $K3$ surfaces. In fact a slope-stable vector bundle on an elliptic curve  has by default no infinitesimal deformations fixing the determinant, while this  certainly does not hold for general vector bundles on abelian surfaces or higher dimensional abelian varieties - it holds exactly for semi homogeneous ones. 

In the present paper we  carry out the strategy outlined above both in the familiar setting of moduli spaces of sheaves on $K3$ surfaces and in the  new setting of moduli spaces of modular vector bundles on HK's of Type $K3^{[2]}$. 
\subsection{Moduli spaces of sheaves on $K3$ surfaces}
\setcounter{equation}{0}
We start by recalling basic definitions and results on moduli of sheaves on a smooth projective polarized surface $(S,h)$ over an algebraically closed field $\KK$. Let   $\cF$  be a  torsion-free sheaf on $S$. A non zero proper subsheaf  $0\not=\cE\subsetneq\cF$ is \emph{GM destabilizing} (where GM stands for Gieseker-Maruyama) if for $m\gg 0$  one has
\begin{equation}\label{gmdestab}
\frac{\chi(S;\cE(mh))}{r(\cE)}\ge \frac{\chi(S;\cF(mh))}{r(\cF)},
\end{equation}
it is \emph{GM desemistabilizing} if the above inequality is strict (for $m\gg 0$).
The sheaf $\cF$ is stable if  no destabilizing subsheaf exists, and it is semistable if  no desemistabilizing subsheaf exists. 

Isomorphism classes of GM stable sheaves are well-behaved, and one gets projective moduli schemes by adding points representing GM semistable sheaves. More precisely, let $\xi=r+\xi_1+\xi_2$, where $r$ is a positive integer, $\xi_1\in  H^{1,1}_{\ZZ}(S)$ and $\xi_2\in H^{4}(S;\ZZ)$. A  classical Theorem of Gieseker and Maruyama~\cite{pippo} states that there exists a quasi-projective coarse moduli scheme (over $\KK$) $\cM_{\xi}(S,h)$ for GM stable torsion-free sheaves $\cF$ on $S$ such that 
\begin{equation}\label{datifascio}
r(\cF)=r,\quad c_1(\cF)=\xi_1,\quad c_2(\cF)=\xi_2.
\end{equation}
 It is often the case that $\cM_{\xi}(S,h)$  is not projective, but it is always an open subscheme of the projective moduli scheme 
$\ov{\cM}_{\xi}(S,h)$ obtained by adding \emph{S-equivalence} classes of torsion-free GM semistable sheaves on $S$ with the same rank and Chern classes. (S-equivalence is a relation weaker than isomorphism.) We remark that  $\cM_{\xi}(S,h)$ might not be dense in $\ov{\cM}_{\xi}(S,h)$. 

Now let us assume that the field $\KK$ has characteristic $0$. If $[\cF]\in\cM_{\xi}(S,h)$ is a point representing a GM stable sheaf $\cF$, then the \emph{expected dimension} of $\cM_{\xi}(S,h)$ at $[\cF]$ is given by
\begin{equation}\label{expdim}
\Delta(\cF)-(r^2-1)\chi(S,\cO_s)+q_S,
\end{equation}
where 
\begin{equation}\label{dueversioni}
\Delta(\cF):=2r c_2(\cF)-(r-1) c_1(\cF)^2=-2r \ch_2(\cF)+\ch_1(\cF)^2.
\end{equation}
is the \emph{discriminant} of $\cF$ (identified with an integer) and $q_S=h^1(S,\cO_S)$ is the irregularity of $S$.  
The reason for naming \lq\lq expected dimension\rq\rq\ the quantity in~\eqref{expdim} is that it equals the actual dimension of 
$\cM_{\xi}(S,h)$ at $[\cF]$ if the trace map 
\begin{equation}\label{traccia}
\Ext^2(\cF,\cF)\overset{\Tr}{\lra} H^2(S,\cO_S)
\end{equation}
 is an isomorphism - by Serrre duality this is equivalent to the hypothesis that every homomorphism $\cF\to\cF\otimes K_S$ is scalar multiplication by a global section of $K_S$.  In fact the Artamkin-Mukai Theorem \cite{artamkin,muksympl} gives that if~\eqref{traccia}  is an isomorphism then  the sheaf is $\cF$ is unobstructed and hence $\cM_{\xi}(S,h)$ is smooth at $[\cF]$ with Zariski tangent space identified with $\Ext^1(\cF,\cF)$ (because the germ of $\cM_{\xi}(S,h)$  at $[\cF]$ is identified with the deformation space of  $\cF$). 

If $S$ is a  surface of general type, and the expected  dimension of $\cM_{\xi}(S,h)$ is non zero,  the moduli space 
$\cM_{\xi}(S,h)$ may very well be empty, or non empty but reducible, or non empty of dimenison greater than  the expected one. 

By contrast, if $S$ is a $K3$ surface then $\cM_{\xi}(S,h)$ behaves extremely well. One beautiful feature is that if $[\cF]\in \cM_{\xi}(S,h)$ then the trace map in~\eqref{traccia} is surjective. In fact by stability every homomorphism $\cF\to\cF$ is scalar multiplication by a constant, and since  $K_S$ is trivial this implies that the map in~\eqref{traccia} is an isomorphism. Hence 
$\cM_{\xi}(S,h)$ is smooth of the expected dimension (we are assuming that it is  non empty). Moreover Serre duality 
$$\Ext^1(\cF,\cF)\times \Ext^1(\cF,\cF)\to H^2(S,\cO_S)$$
 defines a skew-symmetric regular form on $\cM_{\xi}(S,h)$ which is closed. In particular $\cM_{\xi}(S,h)$ has even dimension (again we are assuming that it is  non empty), we let $n(\xi)$ be half its dimension.

We will state a result  which describes more precisely  $\cM_{\xi}(S,h)$ when  GM semistability coincides with GM stability.

First we need to discuss  GM semistability versus GM stability. Let $S$ be a projective $K3$ surface, and let 
$\xi=r+\xi_1+\xi_2$. 
Then there is  a locally finite union of rational of $\xi$-walls in $\Amp(X)_{\RR}$ (a rational wall is the intersection $\lambda^{\bot}\cap \Amp(X)_{\RR}$ where $\lambda\in H^{1,1}_{\ZZ}(S)$), such that if $h$ is a polarization outside the union of the $\xi$-walls then the following holds. Let  $\cF$ be a torsion-free  sheaf such that~\eqref{datifascio} holds, which is 
$h$ GM semistable but not not $h$ GM stable; then $\ch(\cF)$ is not primitive (notice that $\ch(\cF)$ is integral because the intersection form on $S$ is even), i.e.~there exists an integer $m\ge 2$ such that $\frac{\ch(\cF)}{m}$ is integral. In particular, if 
$\ch(\cF)$ is primitive then $\cM_{\xi}(S,h)=\ov{\cM}_{\xi}(S,h)$. A polarization outside the union of the $\xi$-walls is called $\xi$-generic. (See Subsection~\ref{murocella} for the definition of walls.)

Below is a well known fundamental result on moduli spaces of sheaves on $K3$ surfaces.
\begin{thm}[Mukai~\cite{mukvb}, Huybrechts-G\"ottsche~\cite{huygot}, O'Grady~\cite{ogwt2}, Yoshioka~\cite{yoshioka-vbs-on-ell-surfcs,yoshiexpls,yoshitwist}]\label{modonk3}
Let $S$ be a projective $K3$ surface. Let $\xi$ be as above and suppose that $r+\xi_1+\frac{1}{2}\xi_1^2-\xi_2$ is indivisible. 
If $h$ is a $\xi$-generic polarization of $S$ then 
$\cM_{\xi}(S,h)$ is non empty, irreducible, smooth and projective, of the expected dimension
\begin{equation*}
2n(\xi):=2r \xi_2-(r-1)\xi_1^2-2(r^2-1).
\end{equation*}
If $\cM_{\xi}(S,h)$ is not zero dimensional (i.e.~$n(\xi)>0$) then $\cM_{\xi}(S,h)$  is a HK variety deformation equivalent to $S^{[n(\xi)]}$.
\end{thm}
If $r=1$ then $\cM_{\xi}(S,h)$ parametrizes sheaves $\cI_Z\otimes L$ where $L$ is the line bundle such that $c_1(L)=\xi_1$, and  hence $\cM_{\xi}(S,h)$ is   is trivially  isomorphic to $S^{[n(\xi)]}$. The real interest of Theorem~\ref{modonk3} lies in the case $r\ge 2$. If $r\ge 2$ then in general $\cM_{\xi}(S,h)$ is \emph{not}  birational to $T^{[n(\xi)]}$ for any $K3$ surface $T$.

We will go through the proof of Theorem~\ref{modonk3} under the hypothesis that $r+\xi_1$ is  indivisible i.e.~$r$ and the maximum integer dividing $\xi_1$ are coprime. The reason is that the  proof under these hypotheses will be extended to modular sheaves on HK's. The statement of Theorem~\ref{modonk3} for $r+\xi_1$ divisible can be obtained from the result for $r+\xi_1$ indivisible via Mukai reflections, see~\cite{yoshiexpls,yoshitwist}. In this respect we make the following a comment. First Mukai reflections for higher dimensional HK's do not make sense. Secondly a proof of Theorem~\ref{modonk3} 
for $r+\xi_1$ divisible might possibly be obtainable by extending the method adopted for $r+\xi_1$ indivisible. Such an extension might indicate how to prove results for moduli of modular sheaves on  HK varieties (say of Type $K3^{[2]}$) which go beyond those discussed in Theorem~\ref{unicita}. 

\begin{rmk}
If $n(\xi)>0$ there is a beautiful description of the Hodge structure of $H^2(\cM_{\xi}(S,h))$ and its Beauville-Bogomolov-Fujiki quadratic form (due to Mukai)  in terms of the Mukai vector associated to $\xi$. This is a key result, which has proved to be of great relevance for the development of the theory of HK manifolds. We will not discuss its proof.
\end{rmk}
\subsection{Modular sheaves}\label{grandef}
\setcounter{equation}{0}
Let  $\cF$ be a rank $r$ torsion-free sheaf on a manifold $X$. The \emph{discriminant} $\Delta(\cF)\in H^{2,2}_{\ZZ}(X)$ is defined by~\eqref{dueversioni}. 
Below is our key definition. 
\begin{dfn}\label{effemod}
Let $X$ be a HK manifold of dimension $2n$, and let $q_X$ be its  Beauville-Bogomolov-Fuiki (BBF) bilinear symmetric form. A torsion free sheaf $\cF$ on $X$ is \emph{modular} if 
there exists $d(\cF)\in\QQ$ such that
\begin{equation}\label{fernand}
\int_X \Delta(\cF)\smile \alpha^{2n-2}=d(\cF) \cdot (2n-3)!! \cdot q_X(\alpha,\alpha)^{n-1}
\end{equation}
   for all $\alpha\in H^2(X)$.
\end{dfn}
\begin{expl}\label{modsuk3}
If $X$ is a $K3$ surface then the equality in~\eqref{fernand} holds with $d(\cF)=\Delta(\cF)$. Hence every torsion free sheaf on a $K3$ surface is modular. By way of contrast tautological vector bundles on $S^{[n]}$ where $S$ is a $K3$ surface, are not modular in general if $n\ge 2$ - see Example 2.6 in~\cite{ogmodvb}.
\end{expl}
\begin{rmk}\label{semphodge}
Let $X$ be a HK variety of dimension $2n$. Let $D(X)\subset H(X)$  be the image of the map $\Sym H^2(X)\to H(X)$ defined by cup-product. Let $D^i(X):=D(X)\cap H^i(X)$. The  pairing $D^i(X)\times D^{4n-i}(X)\to\CC$ defined by intersection product is non degenerate~\cite{verbcohk,bogcohk,rhag}, hence there is a splitting $H(X)= D(X)\oplus D(X)^{\bot}$, where orthogonality is with respect to the intersection pairing. 
Now let  $\cF$ be a torsion free sheaf on $X$. Then $\cF$ is modular if and only if the orthogonal projection of $\Delta(\cF)$ onto $D^4(X)$ is a multiple of the class $q_X^{\vee}$ dual to $q_X$.
 In particular  $\cF$ is modular if $\Delta(\cF)$ is a multiple of $c_2(X)$. 
\end{rmk}
\begin{rmk}\label{caroverb}
Let $X$ be a  HK of Type $K3^{[2]}$. Then $H(X)= D(X)$ (notation as in Remark~\ref{semphodge}). It follows that  a vector bundle $\cF$ on $X$  is modular if and only if  $\Delta(\cF)$ is a multiple of $c_2(X)$. It follows~\cite{verbhyper} that if $\cF$ is a modular vector bundle, slope-stable for a polarization $h$, then $\End_0(\cF)$ is hyperholomorphic on $(X,h)$, where  $\End_0(\cF)$ is  the vector bundle of traceless endomorphisms of $\cE$. More generally, on an arbitrary HK polarized variety $(X,h)$ there should be 
a relation between the property of being modular and that of being hyperholomorphic. 
\end{rmk}
\begin{rmk}\label{deltanti}
Let $X$ be a HK manifold of dimension $2n$, and let $\cF$ be a torsion free  modular sheaf on $X$.  
Then 
\begin{equation}\label{sbrodolo}
\int\limits_{X}\Delta(\cF)\smile\alpha_1\smile\ldots\smile \alpha_{2n-2}=d(\cF)\cdot\wt{\sum}\, q_X(\alpha_{i_1},\alpha_{i_2})\cdot\ldots\cdot q_X(\alpha_{i_{2n-3}},\alpha_{i_{2n-2}}),
\end{equation}
for all $\alpha_1,\ldots, \alpha_{2n}\in H^2(X)$, where $\wt{\sum}$ means that in the summation we avoid repeating addends which are formally equal (i.e.~are equal modulo reordering of the factors   $q_X(\cdot,\cdot)$'s and switching the entries in $q_X(\cdot,\cdot)$). In fact both sides of the equation in~\eqref{sbrodolo} are multilinear symmetric maps $H^2(X)^{2n-2}\to\CC$, and  by~\eqref{fernand}  they  give the same polynomial when computed on $(\alpha,\alpha,\ldots,\alpha)$. Hence they are both the polarization of the same polynomial, and thus equal. 
\end{rmk}
\subsection{Moduli spaces of modular sheaves on HK's of Type $K3^{[2]}$}\label{risprin}
\setcounter{equation}{0}
We state our recent result~\cite{ogmodvb} on  
modular sheaves with certain discrete invariants on projective HK's of Type $K3^{[2]}$. 

First we recall the discrete invariants indicizing  moduli spaces of polarized 
HK's  of Type $K3^{[2]}$. Let $(X,h)$ be one such polarized HK (we emphasize that the ample class $h\in H^{1,1}_{\ZZ}(X)$ is primitive). Then either
\begin{equation}\label{divuno}
q(h,H^2(X;\ZZ))=\ZZ,\quad q(h)=e>0, \quad e\equiv 0\pmod{2}
\end{equation}
or
\begin{equation}\label{divdue}
q(h,H^2(X;\ZZ))=2\ZZ,\quad q(h)=e>0,\quad e\equiv 6\pmod{8}.
\end{equation}
Conversely, if $e$ is a positive integer which is even (respectively congruent to $6$ modulo $8$) there exists $(X,h)$ 
such that~\eqref{divuno}  (respectively~\eqref{divdue}) holds. 
Let $\cK_{e}^1$ be the moduli space of polarized
 HK's $(X,h)$ of Type $K3^{[2]}$ such that~\eqref{divuno} holds, and let $\cK_{e}^2$ be the moduli space of polarized
 HK's $(X,h)$ of Type $K3^{[2]}$ such that~\eqref{divdue} holds. Both $\cK_{e}^1$ and $\cK_{e}^2$ are irreducible. 
\begin{thm}\label{unicita}
Let $i\in\{1,2\}$ and let $r_0,e$ be positive integers such that  $r_0\equiv i\pmod{2}$ and
\begin{equation}\label{econ}
e\equiv
\begin{cases}
4r_0-10 \pmod{8r_0} & \text{if $r_0\equiv 0 \pmod{4}$,} \\
\frac{1}{2}(r_0-5) \pmod{2r_0} & \text{if $r_0\equiv 1 \pmod{4}$,} \\
-10 \pmod{8r_0}  & \text{if $r_0\equiv 2 \pmod{4}$,} \\
-\frac{1}{2}(r_0+5) \pmod{2r_0}  & \text{if $r_0\equiv 3 \pmod{4}$.}
\end{cases}
\end{equation}
Suppose that $[(X,h)]\in\cK^i_e$  is a generic point. Then up to isomorphism there exists one and only one $h$ slope-stable vector bundle $\cE$ on $X$ such that
 \begin{equation}\label{ele}
r(\cE)=r_0^2,\quad c_1(\cE)=\frac{r_0}{i} h,\quad \Delta(\cE)  =  \frac{r(\cE)(r(\cE)-1)}{12}c_2(X).
\end{equation}
Moreover $H^p(X,End_0\cE)=0$ for all $p$. 
\end{thm}
\begin{rmk}
Let  $\cE$ be a modular torsion-free sheaf on a hyperk\"ahler manifold $X$ of Type $K3^{[2]}$. Then $r(\cE)$ divides the square of a generator of the ideal 
$\{q_X(c_1(\cE),\alpha) \mid \alpha\in H^2(X;\ZZ)\}$, see Proposition~\ref{restrango}. In Theorem~\ref{unicita} we consider the extremal case in which $r(\cE)$ equals  the square of a generator of the ideal defined above.
\end{rmk}
\begin{rmk}
Let $[(X,h)]\in\cK^2_6$ be  generic. Then $(X,h)$ is isomorphic to the variety of lines $F(Y)$ on a generic cubic hypersurface  $Y\subset\PP^5$ polarized by the Pl\"ucker embedding, and  the vector bundle $\cE$  of Theorem~\ref{unicita} with $r_0=2$ is isomorphic to the restriction  of the tautological quotient vector bundle on $\Gr(2,\CC^6)$. Similarly, let $[(X,h)]\in\cK^2_{22}$ be generic. Then $(X,h)$ is isomorphic to the Debarre-Voisin variety 
associated to a generic $\sigma\in \bigwedge^3 V_{10}^\vee$, where 
  $V_{10}$ is a $10$ dimensional complex vector space, and 
\begin{equation}\label{eqDV}
X_\sigma:=\{[W]\in \Gr(6,V_{10}) \mid \sigma\vert_{ W}=0\}.
\end{equation}
 The vector bundle $\cE$  of Theorem~\ref{unicita} with $r_0=2$ is isomorphic to the restriction to $X_\sigma$ of the tautological quotient vector bundle on $\Gr(6,V_{10})$. 
\end{rmk}
\begin{rmk}
The proof of Theorem~\ref{unicita} that we will give provides a blueprint for the proof of similar results for HK varieties of other deformation types. We are working on such results for $4$ dimensional  HK varieties of Kummer type.
\end{rmk}
\subsection{Notation and a few well know results}\label{subsec:ginevra}
\setcounter{equation}{0}
\begin{enumerate}
\item[$\bullet$] 
Algebraic variety is sinonimous of complex quasi projective variety (not necessarily irreducible), unless we state the contrary. 
\item[$\bullet$] 
If  $\pi\colon X\to Y$ is a fibered variety, then  \emph{a} generic fiber of
 $\pi$ is $\pi^{-1}(y)$ for $y$ in  a dense open subset of $Y$, while  \emph{the} generic fiber of
 $\pi$ is the scheme $X\times_{\CC(Y)}$ obtained from $\pi\colon X\to Y$ by base change.  
\item[$\bullet$] 
Let $X$ be a smooth complex quasi projective variety and $\cF$ a  coherent sheaf  on $X$. We only consider  topological  Chern classes   $c_i(\cF)\in H^{2i}(X(\CC);\ZZ)$.
\item[$\bullet$] 
Abusing notation we say that a smooth projective variety $X$ is an  abelian variety if it is isomorphic to the variety underlying an abelian variety $A$. In other words 
$X$ is a torsor of $A$. 
\item[$\bullet$] 
Let $X$ be a HK manifold. We let $q_X$, or simply $q$, be the BBF symmetric bilinear form of $X$. We recall that $q_X$ is strictly positive on K\"ahler classes.
\item[$\bullet$] 
Let $X$ be a HK manifold of dimension $2n$.
We let $c_X$ be the \emph{normalized Fujiki constant of $X$}, i.e.~the rational positive number such that for all $\alpha\in H^2(X)$ we have
\begin{equation}\label{relfuj}
\int\limits_{X}\alpha^{2n}=c_X\cdot (2n-1)!!\cdot q_X(\alpha)^n.
\end{equation}
A hyperk\"ahler (HK) \emph{variety} is a projective compact HK manifold. 
\item[$\bullet$] 
Let 
\begin{equation}\label{mappamu}
\mu\colon H^2(S)\to H^2(S^{[n]}) 
\end{equation}
be the composition of the  natural symmetrization map $H^2(S)\to H^2(S^{(n)})$ and the pull-back $H^2(S^{(n)})\to H^2(S^{[n]})$ defined by  the Hilbert-Chow map $S^{[n]}\to S^{(n)}$.
\item[$\bullet$] 
Let $(X,h)$ be an irreducible polarized projective variety. 
Let $\cF$ be a torsion-free sheaf on $X$. A  subsheaf  $\cE\subset\cF$ is \emph{slope-destabilizing} if $0<r(\cE)<r(\cF)$ and $\mu_h(\cE)\ge \mu_h(\cF)$, where $r(\cE),r(\cF)$ are the ranks of $\cE,\cF$, and $\mu_h(\cE),\mu_h(\cF)$ are the $h$-slopes of $\cE,\cF$.   If $\mu_h(\cE)> \mu_h(\cF)$ then 
$\cE\subset\cF$ is \emph{slope-desemistabilizing}. We use similar terminology for exact sequences $0\to\cE\to\cF\to \cG\to 0$. The sheaf $\cF$ is 
\emph{slope-stable} if it has no destabilizing subsheaf, and it is \emph{slope-semistable} if it has no desemistabilizing subsheaf. The slope of a torsion-free sheaf $\cF$ on an irreducible curve  does not depend on the polarization: we will denote it 
by $\mu(\cF)$.
\item[$\bullet$] 
A torsion-free sheaf  on a polarized variety $(X,h)$ is \emph{strictly} $h$ slope-semistable if it is $h$ slope-semistable but not $h$ slope-stable. 
\item[$\bullet$] 
Let $\cF$ be a torsion-free sheaf on  $(X,h)$. A  subsheaf  $0\not=\cE\subsetneq\cF$ is \emph{GM destabilizing} (GM stands for Gieseker-Maruyama) if  $p_{\cE}(k)\ge p_{\cF}(k)$ for $k>>0$, where for a sheaf $\cG$ of non zero rank $p_{\cG}(k):=\frac{\chi(X,\cG\otimes\cO_X(k))}{r(\cG)}$ is the \emph{normalized} Hilbert polynomial. The  subsheaf  $\cE$ is \emph{GM desemistabilizing}  if  $p_{\cE}(k)> p_{\cF}(k)$ for $k>>0$. The sheaf $\cF$ is 
\emph{GM stable} if it has no destabilizing subsheaf, and it is \emph{GM semistable} if it has no desemistabilizing subsheaf. The moduli space of $S$-equivalence classes of GM torsion-free sheaves with a fixed Chern character is a projective scheme (a GM stable sheaf is $S$-equivalent to a GM semistable sheaf only if they are  isomorphic). Slope-stability implies GM stability, and GM semistability implies slope-semistability. In particular the moduli space of slope-stable sheaves is an open subscheme of the moduli space of GM semistable sheaves.  If $X$ is an irreducible curve (semi)stability is independent of the polarization, and therefore we will make no mention of the polarization. 
\end{enumerate}

\newpage

\section{Variation of stability for modular  sheaves}\label{camere}
\setcounter{equation}{0}
\subsection{Background and overview}\label{backover}
\setcounter{equation}{0}
Let $X$ be an irreducible smooth projective variety, let $\Amp(X)\subset H^{1,1}_{\ZZ}(X)$ be the ample cone, 
and let $\Amp(X)_{\RR}\subset H^{1,1}_{\RR}(X)$ be the real convex hull of the ample cone. The moduli space of $h$ slope-stable 
torsion-free sheaves on $X$ with fixed Chern character $\xi$ depends on the ray spanned by $h$, and hence the question: 
 how does the moduli space vary when $h$ changes?

If $X$ is a surface there is a  decomposition of $\Amp(X)_{\RR}$ in chambers and walls which  gives a first answer to the question above. More precisely, let $r$ and $\Delta$ be the rank and discriminant of sheaves with Chern character $\xi$. A \emph{$\xi$-wall} is given by $\lambda^{\bot}\cap  \Amp(X)_{\RR}$, where 
\begin{equation}\label{corsofrancia}
-\frac{r^2\Delta}{4}\le \int_X \lambda^2<0,\quad \lambda\in H^{1,1}_{\ZZ}(X).
\end{equation}
The set of $\xi$-walls is locally finite and hence the complement of their union is an open subset of  $\Amp(X)_{\RR}$. An \emph{open chamber} is a connected component of the complement (in  $\Amp(X)_{\RR}$) of the union of the $\xi$-walls. The first answer to the question asked above is that for $h\in\cC$ the moduli space of $h$ slope-stable torsion-free sheaves on $X$ with Chern character $\xi$ is independent of  $h$. More precisely this means that if $h_1,h_2\in\cC$ and $\cF$ is a torsion-free sheaf on $X$  such that $\ch(\cF)=\xi$, then $\cF$ is $h_1$ slope-stable if and only if it is $h_2$ slope-stable.

If $X$ is a general variety of dimension greater than $2$ then  the picture is substantially more complex, see for example~\cite{grebmaster}. 

A key observation of~\cite{ogmodvb} is that one gets a similar picture if $X$ is a HK variety and the sheaves that we consider are modular. More precisely, we get an  analogous result if we replace the intersection form on $H^2(X)$ by the BBF quadratic form, and we make a suitable modification 
of the lower bound in~\eqref{corsofrancia} (we replace $\Delta$ by $d(\cF)$ and   we introduce Fujiki's constant $c_X$ in the denominator). We review this result in Subsection~\ref{murocella}, and we sketch the proof.

Next, suppose that $X$ is a HK variety with a Lagrangian fibration  $\pi\colon X\to\PP^n$. Then $f:=\pi^{*}c_1(\cO_{\PP^n}(1))$ is in the closure of the  ample cone. The wall and chamber decomposition of $\Amp(X)_{\RR}$ allows to give a quantitative version of  the principle \lq\lq If a polarization $h$ is close to $f$, then $h$ slope-stability of a sheaf $\cF$ on $X$ is related to slope stability of the restriction of $\cF$ to a generic Lagrangian fiber\rq\rq.  This result, a key ingredient in the proof of the main results of the present paper, is presented in Subsection~\ref{subsec:pazzia}.

\subsection{Walls and chambers decomposition for a modular sheaf}\label{murocella}
\setcounter{equation}{0}
\begin{dfn}
Let $a$ be a positive real number. An \emph{$a$-wall} of $\Amp(X)_{\RR}$ is the intersection 
$\lambda^{\bot}\cap \Amp(X)_{\RR}$, where 
  $\lambda\in H^{1,1}_{\ZZ}(X)$, 
$ -a \le q_X(\lambda)< 0$,
and orthogonality is with respect to the BBF quadratic form $q_X$.
\end{dfn}
As is well-known, the set of $a$-walls  is  locally finite, in particular the union of all the $a$-walls  is closed in $\Amp(X)_{\RR}$. 
\begin{dfn}
An \emph{open $a$-chamber} is a connected component of the complement (in $\Amp(X)_{\RR}$) of   the union of all the $a$-walls. 
\end{dfn}
\begin{rmk}\label{convesso}
An open $a$-chamber is convex.
\end{rmk}
\begin{dfn}\label{adieffe}
Let $X$ be a HK manifold, and let $\cF$ be a modular torsion free sheaf  on $X$.
Then
\begin{equation}\label{esmeralda}
a(\cF):=\frac{r(\cF)^2 \cdot d(\cF) }{4c_X},
\end{equation}
where $d(\cF)$ is as in Definition~\ref{effemod}.
\end{dfn}
\begin{expl}\label{aeffek3}
Let $X$ be a $K3$ surface. Then every torsion free sheaf $\cF$ on $X$ is modular, see Example~\ref{modsuk3}. Since $c_X=1$ and 
$d(\cF)=\Delta(\cF)$, we have
\begin{equation*}
a(\cF):=\frac{r(\cF)^2 \cdot \Delta(\cF) }{4}.
\end{equation*}
\end{expl}
Below is the main result that we discuss in the present subsection.
\begin{prp}\label{campol}
Let $X$ be a  HK variety of dimension $2n$, and let  $\cF$ be a   torsion free modular sheaf  on $X$. Then the following hold:
\begin{enumerate}
\item
Suppose that $h$ is an ample divisor class on $X$ which belongs to an open $a(\cF)$-chamber. If  $\cF$ is  strictly $h$ slope-semistable  there exists an exact sequence
of torsion free non zero sheaves
\begin{equation}\label{farnesina}
0\lra \cE\lra \cF\lra \cG\lra 0
\end{equation}
such that $r(\cF) c_1(\cE)-r(\cE) c_1(\cF)=0$.
\item
Suppose that $h_0,h_1$ are ample divisor classes on $X$  belonging to the same open $a(\cF)$-chamber. Then $\cF$ is 
 $h_0$ slope-stable if and only if it is $h_1$ slope-stable. 
\end{enumerate}
\end{prp}
Item~(1) of Proposition~\ref{campol} gives that if $\cF$ is strictly slope-semistable for a polarization belonging to an open $a(\cF)$-chamber, then  it is not $h$ slope-stable for \emph{any} polarization $h$, because the exact sequence in~\eqref{farnesina} is slope-destabilizing. Item~(2) of Proposition~\ref{campol} is the result that we presented in Subsection~\ref{backover}.

We proceed to sketch the proof of Proposition~\ref{campol}. First we introduce a piece of notation.  Let 
$\cE,\cF$  be sheaves on an irreducible smooth  variety $X$. We let
\begin{equation}\label{pecora}
\lambda_{\cE,\cF}:=(r(\cF) c_1(\cE)-r(\cE) c_1(\cF))\in H^2(X;\ZZ).
\end{equation}
The lemma below, which follows from Fujiki's relation,  shows that as far as  slope-(semi)stability on a   HK variety  is concerned, the BBF form plays the r\^ole of the intersection form on a surface. We emphasize that in the following lemma we do not assume that the sheaves are modular.
\begin{lmm}\label{comesup}
Let $(X,h)$ be a polarized HK variety, and let $\cE,\cF$  be non zero torsion free sheaves on  $X$.
Then
\begin{enumerate}
\item[(a)]
$\mu_h(\cE)>\mu_h(\cF)$ if and only if $q_X(\lambda_{\cE,\cF},h)>0$.
\item[(b)]
$\mu_h(\cE)=\mu_h(\cF)$ if and only if $q_X(\lambda_{\cE,\cF},h)=0$.
\end{enumerate}
\end{lmm}
\begin{proof}
Let  $2n$  be the dimension of $X$. 
Fujiki's relation~\eqref{relfuj} is equivalent to the validity, for all $\alpha_1,\ldots, \alpha_{2n}\in H^2(X)$, of the equality
\begin{equation}\label{polrel}
\int\limits_{X}\alpha_1\smile\ldots\smile \alpha_{2n}=c_X\cdot\wt{\sum}\, q_X(\alpha_{i_1},\alpha_{i_2})\cdot\ldots\cdot q_X(\alpha_{i_{2n-1}},\alpha_{i_{2n}}),
\end{equation}
  where $\wt{\sum}$ means that in the summation we avoid repeating addends which are formally equal i.e.~are equal modulo reordering of the factors   $q_X(\cdot,\cdot)$ and switching the entries in the factors $q_X(\cdot,\cdot)$.

We have $\mu_h(\cE)>\mu_h(\cF)$ if and only if 
$\int_X\lambda_{\cE,\cF}\smile h^{2n-1}>0$, and  by~\eqref{polrel} this holds if and only if 
\begin{equation*}
c_X\cdot(2n-1)!!\cdot q_X(\lambda_{\cE,\cF},h)\cdot q_X(h)^{n-1}>0.
\end{equation*}
Item~(a) follows, because $c_X>0$ and $q_X(h)>0$. 

 We have $\mu_h(\cE)=\mu_h(\cF)$ if and only if 
$\int_X\lambda_{\cE,\cF}\smile h^{2n-1}=0$, and hence   Item~(b) follows again by Fujiki's formula.
\end{proof}
The next proposition is again valid for arbitrary torsion-free sheaves on a HK variety, and the proof, thanks to Lemma~\ref{comesup}, is a replica of the analogous statement valid for sheaves on an arbitrary (smooth projective) surface, see for example the proof of Lemma~4.C.5 in~\cite{huylehnbook}.
\begin{prp}\label{polint}
Let $X$ be a HK variety, and let $h_0,h_1$ be ample divisor classes on $X$. Suppose that $\cF$ is a torsion free sheaf on $X$ which is $h_0$ slope-stable and not $h_1$ slope-stable. Then there exists $h\in(\QQ_{+}h_0+\QQ_{+}h_1)$ such that $\cF$ is strictly $ h$ slope-semistable, i.e.~$\cF$ is $h$ slope-semistable but not $h$ slope-stable. 
\end{prp}
The result below, motivated by Proposition~\ref{polint}, is valid for \emph{modular}  torsion free sheaves. In fact this is the only instance in which the modularity hypothesis is needed for the proof of Proposition~\ref{campol}.
\begin{prp}\label{propsemi}
Let $(X,h)$ be a polarized HK variety of dimension $2n$.   
Let $\cF$ be  a    torsion free modular \emph{strictly}  $h$ slope-semistable  sheaf on $X$,  and let 
\begin{equation}\label{faria}
0\lra \cE\lra \cF\lra \cG\lra 0
\end{equation}
be an   exact sequence of non zero torsion free sheaves which is $h$ slope destabilizing, i.e.~$\mu_h(\cE)=\mu_h(\cF)$. Then
\begin{equation}\label{doppelgang}
-a(\cF) \le q_X(\lambda_{\cE,\cF})\le 0.
\end{equation}
Moreover $q_X(\lambda_{\cE,\cF})= 0$ only if $\lambda_{\cE,\cF}=0$.  
\end{prp}
\begin{proof}
Since the exact sequence in~\eqref{faria} is destabilizing,  
$q_X(\lambda_{\cE,\cF}, h)=0$ by Lemma~\ref{comesup}. Since the BBF form on $\NS(X)$ has signature $(1,\rho(X)-1)$, it follows that 
$q_X(\lambda_{\cE,\cF})\le 0$ with equality only if $\lambda_{\cE,\cF}=0$.  (Recall that $q_X(h)> 0$, because $h$ is ample.)

We are left with proving the second inequality in~\eqref{doppelgang}. Hence we assume that $ q_X(\lambda_{\cE,\cF})< 0$. 
By additivity of the Chern character and by~\eqref{dueversioni}, we have
\begin{equation}\label{lungaeq}
r(\cF) \cdot r(\cG) \Delta(\cE)+ r(\cF) \cdot r(\cE) \Delta(\cG)=r(\cE)\cdot r(\cG)\Delta(\cF)+\lambda_{\cE,\cF}^2.
\end{equation}
 Cupping both sides of the equality in~\eqref{lungaeq} by $h^{2n-2}$, and integrating, we get (here we use the hypothesis that  
$\cF$ is modular)
\begin{multline}\label{pharaon}
\int_X r(\cF) \cdot r(\cG) \Delta(\cE)\smile h^{2n-2}+\int_X  r(\cF) \cdot r(\cE) \Delta(\cG))\smile h^{2n-2}=\\
=r(\cE)\cdot r(\cG)\cdot d(\cF)\cdot(2n-3)!! q_X(h)^{n-1} +c_X\cdot q_X(\lambda_{\cE,\cF})\cdot (2n-3)!! q_X(h)^{n-1}. \\
\end{multline}
By hypothesis $\mu_h(\cE)=\mu_h(\cF)=\mu_h(\cG)$. Since  $\cF$ is $h$ slope-semistable it follows that $\cE$ and $\cG$ are $h$ slope-semistable torsion free sheaves. Thus 
$$\int_X \Delta(\cE)\smile h^{2n-2}\ge0 ,\quad \int_X \Delta(\cG)\smile h^{2n-2}\ge0 $$
 by Bogomolov's inequality, and hence~\eqref{pharaon} gives
\begin{equation}\label{carconte}
-r(\cE)\cdot r(\cG)\cdot d(\cF) \le c_X\cdot q_X(\lambda_{\cE,\cF}).
\end{equation}
Dividing by $c_X$ (which is strictly positive), we see that the second inequality in~\eqref{doppelgang} follows from~\eqref{carconte} and the inequality $r(\cE)\cdot r(\cG)\le r(\cF)^2/4$. 
\end{proof}
We are ready to prove Proposition~\ref{campol}. Item~(1) follows  from Proposition~\ref{propsemi}. In order to prove Item~(2) it suffices to show that if $\cF$ is $h_0$ slope-stable, then it is $h_1$ slope-stable. Suppose that $\cF$ is not $h_1$ slope-stable. By Proposition~\ref{polint}, there exists $h\in(\QQ_{+}h_0+\QQ_{+}h_1)$ such that $\cF$ is strictly $ h$ slope-semistable. Hence there exists an $h$ destabilizing
\begin{equation*}
0\lra \cE\lra \cF\lra \cG\lra 0
\end{equation*}
   exact sequence of non zero torsion free  sheaves.  Since $h_0,h_1$ belong to the same open $a(\cF)$ chamber, also $h$ belongs to the same open $a(\cF)$-chamber, see Remark~\ref{convesso}. Thus, by Proposition~\ref{propsemi},  we get that $\lambda_{\cE,\cF}=0$. It follows that $\cF$ is not $h_0$ slope-stable, and that is a contradition.
\qed
\subsection{Stability of modular sheaves on a Lagrangian HK}\label{subsec:pazzia}
\setcounter{equation}{0}
Let $X$ be a HK manifold with  a surjection $\pi\colon X\to Y$ with connected fibers, where  $Y$ is a K\"ahler manifold such that $0<\dim Y<\dim X$. Then by Matsushita~\cite{mat-fiber-struct,mat-addendum,mat-equidim} and Hwang~\cite{hwang-base-lagr-fibr} the following hold:  $Y$ is  isomorphic to a projective space,  a generic fiber of $\pi$ is an abelian variety, and all fibers of 
$\pi$ are Lagrangian subspaces of $X$ (in particular  $2\dim Y=\dim X$ and the dimension of each fiber  of $\pi$ is half the dimension of $X$). 

For our purposes a Lagrangian fibration on a HK manifold of dimension $2n$ is    a surjection $\pi\colon X\to \PP^n$ with connected fibers. A Lagrangian fibration on a $K3$ surface is nothing else but an elliptic fibration.  Lagrangian fibrations on HK's of higher dimension behave very much like elliptic fibrations on $K3$ surfaces.
\begin{expl}
Let $S$ be a $K3$ surface with an elliptic fibration $\rho\colon S\to\PP^1$. The composition  $S^{[n]}\to S^{(n)}\to  (\PP^1)^{(n)}\cong\PP^n$ is a Lagrangian fibration 
$\pi\colon S^{[n]}\to \PP^n$. A generic fiber of $\pi$ is isomorphic to $C_1\times\ldots\times C_n$, where $C_1,\ldots, C_n$ are generic distinct fibers of  $\rho$. The generic deformation of the couple $(S^{[n]},\pi)$ is \emph{not} obtained by deforming $(S,\rho)$. In fact by~\cite{mat-iso} the deformation space of  $(S^{[n]},\pi)$ is smooth and it has dimension one greater than the deformation space of  $(S,\rho)$.
\end{expl}
\begin{rmk}\label{eccoteta}
Let $\pi\colon X\to \PP^n$ be a Lagrangian fibration on a HK manifold. For $t\in\PP^n$ we let $X_t:=\pi^{-1}(t)$ be the schematic fiber over $t$. If $X_t$ is smooth the image of the restriction map $H^2(X;\ZZ)\to H^2(X_t;\ZZ)$ has rank one, and is generated by  an ample 
class $\theta_t\in H^{1,1}_{\ZZ}(X_t)$, see~\cite{wieneck1}.  If $\cF$ is a sheaf on $X_t$ slope-(semi)stability of $\cF$ will always mean $\theta_t$ slope-(semi)stability.
\end{rmk}
If  $\pi\colon X\to\PP^n$ is  a Lagrangian fibration we let 
 \begin{equation}\label{eccoeffe}
f:=c_1(\pi^{*}\cO_{\PP^n}(1))\in H^{1,1}_{\ZZ}(X).
\end{equation}
 Since $f$ is nef, it belongs to the closure of $\Amp(X)_{\RR}$. As is well-known $q_X(f)=0$.

In order to establish a relation between slope-(semi)stability of a sheaf on a  Lagrangian fibration and slope-(semi)stability of its restriction to a generic Lagrangian fiber we need a definition which extends a  notion  which is very useful when analyzing vector bundle on fibered surfaces, see for example Definition~2.1
 in~\cite{friedman-rk2-ell-surfcs}. 
\begin{dfn}\label{suipol}
Let $X$ be a HK variety equipped with a Lagrangian fibration $\pi\colon X\to\PP^n$. Let $a$ be positive integer. An ample divisor class $h$ on $X$ is 
$a$-\emph{suitable}  if the following holds. Let $\lambda\in H^{1,1}_{\ZZ}(X)$ be a class such that $-a\le q_X(\lambda)< 0$: then either $q_X(\lambda,h)$ and $q_X(\lambda,f)$ have the same sign, or they are both zero. 
\end{dfn}
Notice that the notion of $a$-suitable depends on the chosen Lagrangian fibration. 
\begin{rmk}\label{opportuno}
Let $X$ be a HK variety  with a Lagrangian fibration $\pi\colon X\to\PP^n$, and suppose that $X$ has Picard number $2$, i.e.~$h^{1,1}_{\ZZ}(X)=2$. Given $a>0$ there exists a finite set of $a$-walls because the restriction of $q_X$ to the rank-$2$ lattice $H^{1,1}_{\ZZ}(X)$ represents $0$. Hence there exists one and only one open $a$-chamber $\cC$ such that its closure contains $f$ (recall that $f$ is nef).  A polarization is $a$-suitable if and only if it belongs to $\cC$.
\end{rmk}
Below is the result relating slope-(semi)stability of a sheaf on a  Lagrangian fibration and slope-(semi)stability of its restriction to a generic Lagrangian fiber.
\begin{prp}\label{lagstab}
Let $\pi\colon X\to\PP^n$ be a  Lagrangian fibration of a HK variety  
of dimension $2n$. Let $\cF$ be  a torsion free modular sheaf   on $X$ such that $\sing \cF$ does \emph{not} dominate $\PP^n$. Let $h$ be an ample divisor class on $X$ which is $a(\cF)$-suitable.
   Then the following hold:
\begin{enumerate}
\item[(i)]
If the restriction of $\cF$ to a generic fiber of $\pi$   is  slope-stable,  then $\cF$ is $h$ slope-stable. 
\item[(ii)]
If $\cF$ is $h$ slope-stable then the restriction of $\cF$ to the generic fiber of $\pi$   is  slope-semistable.
\end{enumerate}
\end{prp}
The observation that allows us to prove Proposition~\ref{lagstab} is the following. 
\begin{lmm}\label{contreffe}
Let $X$ be a HK variety  of dimension $2n$ 
equipped with a Lagrangian fibration $\pi\colon X\to\PP^n$, and let $f:=c_1(\pi^{*}\cO_{\PP^n}(1))$. 
Let $\cF$ be a torsion free sheaf on $X$, and let $\cE\subset\cF$ be a  subsheaf with $0<r(\cE)<r(\cF)$. Then the following hold:
\begin{enumerate}
\item[(a)]
If, for generic $t\in\PP^n$, the restriction $\cF_t:=\cF_{|_{X_t}}$ is slope-stable, then
\begin{equation}\label{intneg}
q_X(\lambda_{\cE,\cF},f)<0.
\end{equation}
\item[(b)]
If, for generic $t\in\PP^n$, the subsheaf $\cE_t:=\cE_{|_{X_t}}\subset\cF_t$ is  slope  desemistabilizing, then
\begin{equation}\label{intpos}
q_X(\lambda_{\cE,\cF},f)> 0.
\end{equation}
\end{enumerate}
\end{lmm}
\begin{proof}
Let $h_t:=h_{|X_t}$. We have
\begin{equation*}
\int_{X_t} \lambda_{\cE_t,\cF_t}\smile h_t^{n-1}=\int_X \lambda_{\cE,\cF}\smile h^{n-1}\smile f^n 
=n! c_X \cdot q_X(h,f)^{n-1} \cdot q_X(\lambda_{\cE,\cF},f).
\end{equation*}
In fact the first equality holds because  $f^n$ is the Poincar\`e dual of any fiber of the Lagrangian fibration, and the second equality holds by~\eqref{polrel} and  
 $q_X(f)=0$. Items~(a) and~(b) follow because $c_X$ and $q_X(h,f)$ are strictly positive.
\end{proof}
For the proof of Proposition~\ref{lagstab} see~\cite{ogmodvb}. The proof is similar to the proof of the analogous   result valid for fibered surfaces, see for example~\cite{friedman-rk2-ell-surfcs,ogwt2,yoshioka-vbs-on-ell-surfcs}.
\section{The Mukai lattice and sheaves on $K3$ surfaces}\label{secmuklatt}
\subsection{The Mukai lattice}
\setcounter{equation}{0}
Let $S$ be a $K3$ surface. The \emph{Mukai lattice} of $S$ is the full integral cohomology group $H(S;\ZZ)$ equipped with
the   \emph{Mukai pairing} 
$$\la (r,\ell,s)\,,(r',\ell',s')\ra:=\int_S(\ell\cup\ell'-rs'-r' s),$$
where $r,r'\in H^0(S;\ZZ)$, $\ell,\ell'\in H^2(S;\ZZ)$ and $s,s'\in H^4(S;\ZZ)$.
 Notice that the Mukai pairing is an even bilinear symmetric form. 
Moreover it has  the following key property: if  $\cF,\cE$ 
are sheaves on $S$ then
\begin{equation}\label{mukpair}
\la v(\cF),v(\cE)\ra=-\chi(\cF,\cE):=-\sum_{i=0}^2(-1)^i\dim\Ext^i(\cF,\cE).
\end{equation}
One gives a weight $2$ integral Hodge structure (HS) on $H(S;\CC)$ as follows: it is the direct sum of the standard HS on $H^2(S)$, and the pure HS of type $(1,1)$ on each of $H^0(S)$ and $H^4(S)$.  
\subsection{Dimension and smoothness of moduli spaces of of sheaves on $K3$'s}
\setcounter{equation}{0}
Let $S$ be a $K3$ surface and  $\cF$ be a sheaf on $S$.  The \emph{Mukai vector of $\cF$}  is 
\begin{equation}
v(\cF):=\ch(\cF)\td(S)^{1/2}=(r(\cF),c_1(\cF),\ch_2(\cF)+r(\cF)).
\end{equation}
Our notation for moduli spaces of sheaves on $S$ is the following. Let 
\begin{equation}\label{eccovi}
\vv=(r,\ell,s)\in H(S;\ZZ),\quad r>0,\ \ \ell\in \NS(S).
\end{equation}
 We let $\cM_{\vv}(S,h)$ be the moduli space of GM semistable torsion-free sheaves on $S$ with Mukai vector $\vv$. 
If $\cF$ is a GM semistable torsion-free sheaf on $S$ with $v(\cF)=\vv$ we let $[\cF]\in \cM_{\vv}(S,h)$ be the point representing the $S$-equivalence class of $\cF$.

Suppose that $[\cF]\in \cM_{\vv}(S,h)$ represents a stable sheaf.  Then the germ of $\cM_{\xi}(S,h)$  at $[\cF]$ is identified with the deformation space of  $\cF$. Since $\cF$ is simple we have
$\dim\Ext^0(\cF,\cF)=1$, and by Serre duality it follows that $\dim\Ext^2(\cF,\cF)=1$. By Artamkin-Mukai~\cite{artamkin,muksympl} it follows that
 $\cM_{\vv}(S,h)$  is smooth at $[\cF]$ with tangent space isomorphic to $\Ext^1(\cF,\cF)$. 
The dimension of the latter space is equal (see~\eqref{mukpair}) to
\begin{equation}\label{extuno}
\dim\Ext^1(\cF,\cF)=2+\la v(\cF),v(\cF)\ra=\Delta(\cF)-2(r^2-1).
\end{equation}
Summarizing we have the following result.
\begin{thm}[Artamkin-Mukai~\cite{artamkin,muksympl}]\label{dimemme}
Let $\cF$ be a GM \emph{stable} torsion-free sheaf on $S$ with $v(\cF)=\vv$. Then  $\cM_{\vv}(S,h)$ is smooth of dimension 
$\vv^2+2$ at $[\cF]$. In particular $\vv^2\ge -2$, i.e.~$\Delta(\cF)\ge 2(r^2-1)$.
\end{thm}
\section{Moduli of stable sheaves on elliptic $K3$ surfaces}
\subsection{The main result}
\setcounter{equation}{0}
In the present section we go through the proof of Theorem~\ref{modonk3} for $S$ an elliptic $K3$, 
under a suitable hypothesis on $r$ and $\xi_1$ -this is Theorem~\ref{vbk3ell} below.
\begin{ass}\label{ellass}
$S$ is a projective $K3$ surface with  an elliptic fibration 
\begin{equation}
\pi\colon S\to\PP^1, 
\end{equation}
and the Picard number of $S$ is $2$. 
\end{ass}
We emphasize that in general $\pi$ does not have a section, i.e.~the curve over $\CC(\PP^1)$ obtained by base change has genus $1$ but is not, strictly speaking, an elliptic curve.
An \emph{elliptic fiber} is a fiber of $\pi$. We let $f:=\pi^{*}c_1(\cO_{\PP^1}(1))$.
\begin{rmk}
If Assumption~\ref{ellass} holds the discriminant of the intersection form restricted on $\NS(S)$ is the negative of a square number. Conversely, given $0\not=m\in\ZZ$  there exist such surfaces with discriminant equal to $(-m^2)$. 
\end{rmk}
Let $\vv$ be a Mukai vector as in~\eqref{eccovi}.  We let 
\begin{equation}\label{enea}
n(\vv):=\frac{\vv^2}{2}+1,\quad a(\vv):=\frac{r^2(\vv^2+2r^2)}{4}.
\end{equation}
Recall that the Mukai pairing is even, and hence $n(\vv)$ is an integer. If $\cF$ is an $h$ stable sheaf on $S$ with $v(\cF)=\vv$, then  by Theorem~\ref{dimemme} $\cM_{\vv}(S,h)$ has dimension $2n(\vv)$ at $[\cF]$. 
If $\cF$ is a sheaf such that $v(\cF)=\vv$ then $a(\cF)=a(\vv)$, see Example~\ref{aeffek3} and the second equation in~\eqref{extuno}.
\begin{thm}[O'Grady~\cite{ogwt2},Yoshioka~\cite{yoshioka-vbs-on-ell-surfcs}]\label{vbk3ell}
Let $S$ be an elliptic $K3$ surface as above. Let $\vv\in H(S;\ZZ)$ be as in~\eqref{eccovi}.
Suppose that $\vv^2\ge -2$ and that $r$ is coprime to $q_S(f,\ell)$.  If $h$ is an $a(\vv)$-suitable polarization of $S$ then  $\cM_{\vv}(S,h)$ is non empty, irreducible, smooth of dimension 
$\vv^2+2$, and birational to $S^{[n(\vv)]}$. If $n(\vv)>0$ then  $\cM_{\vv}(S,h)$  is a HK variety of Type $K3^{[n(\vv)]}$.
\end{thm}
\begin{rmk}
In~\cite{ogwt2} the result is proved under the hypothesis that $q_S(f,\ell)=1$. This is sufficient to obtain  Theorem~\ref{casoprim}, see Remark~\ref{rmk:puzzone}. 
\end{rmk}

\subsection{GM semistability is the same as slope stability, and restriction to a generic elliptic fiber is stable}\label{subsec:algardi}
\setcounter{equation}{0}
Since the Picard number of $S$ is $2$ and $h$ is $a(\vv)$-suitable, it belongs to an open $a(\vv)$-chamber - see Remark~\ref{opportuno}. It follows that every sheaf parametrized by  $\cM_{\vv}(S,h)$ is $h$ slope-stable. In fact assume that $[\cF]\in \cM_{\vv}(S,h)$ and $\cF$ is not  $h$ slope-stable. Then $\cF$ is strictly 
$h$ slope-semistable because it is $h$ GM semistable. By Proposition~\ref{campol} it follows that there exists a subsheaf $\cE\subset\cF$ such that $0<r(\cE)<r$ and $r c_1(\cE)=r(\cE)\ell$. Intersecting with $f$ we get that 
$$r \cdot q_S(f,c_1(\cE))=r(\cE) \cdot q_S(f,\ell),$$
 and this contradicts the hypothesis  that $r$ is coprime to $q_S(f,\ell)$.

Let $\cF$ be  a torsion free sheaf  with $v(\cF)=\vv$. By Proposition~\ref{lagstab} $\cF$ is $h$ slope-stable if and only if it restricts to a stable vector bundle  on a generic elliptic fiber $C_t$, because since $r$ is coprime to $q_S(f,\ell)$ there is no strictly semistable vector bundle on $C_t$ of rank $r$ and degree $q_S(f,\ell)$. 
\subsection{Twisting the Mukai vector}\label{subsec:torcere}
\setcounter{equation}{0}
 By Subsection~\ref{subsec:algardi} every sheaf parametrized by $\cM_{\vv}(S,h)$ is $h$ slope-stable, and hence it remains $h$ slope-stable after tensorization with any line bundle on $S$. It follows that we have an isomorphism of moduli schemes
\begin{equation}
\begin{matrix}
\cM_{\vv}(S,h) & \overset{\sim}{\lra} & \cM_{\ww_m}(S,h) \\
[\cF] & \mapsto & [\cF\otimes\cO_S(mF)]
\end{matrix}
\end{equation}
where $F$ is an elliptic fiber and  $\ww_m:=\vv\smile e^{mf}$.  
\begin{lmm}
Let  $\ww=(r,\ell+xf,t)$, where $x,t$ are integers. If $\ww^2=\vv^2$ then there exists $m\in\ZZ$ such that $\ww=\vv\cup e^{mf}$.
\end{lmm}
\begin{proof}
Let us show that $r$ divides $x$. Since 
$$\ell^2+2x\ell\smile f-2rt=\ww^2=\vv^2=   \ell^2-2rs,$$
we get that $r$ divides $x\ell\smile f$, and hence $r\mid x$ because  $r$ is coprime to $q_S(f,\ell)$. Let $m:=\frac{x}{r}$. Then 
$$\vv\cup e^{mf}=r+(\ell+x f)+s\eta+m\ell\smile f.$$
Multiplication by $e^{\alpha}$ is an isometry of the Mukai lattice for any $\alpha\in H^2(S;\ZZ)$, hence $\vv\cup e^{mf}$ has  square equal to $\vv^2$, which is equal to $\ww^2$ by hypothesis.  Since the degree $0$ and degree $2$ components of $\vv\cup e^{mf}$ are equal to the corresponding components of $\ww$, it follows that  $\ww=\vv\cup e^{mf}$.
\end{proof}
\subsection{Existence of a stable vector bundle}\label{subsec:innocenzox}
\setcounter{equation}{0}
We prove the following result.
\begin{prp}\label{fibvettgen}
There exists a vector bundle  $\cE$ on $S$ such that $c_1(\cE)=\ell+xf$ for some $x\in\ZZ$ and 
 the restriction  of $\cE$ to a generic fiber is stable.
\end{prp}
First we recall results on vector bundles on curves of genus $1$. Let $\cC$ be a geometrically irreducible smooth curve of genus $1$ over a  field $K$. We do not assume that $K$ is algebraically closed.  
\begin{thm}[Atiyah~\cite{atiell}]\label{fibell}
Let $r$ be a positive integer and $\cL$ be a line bundle on $\cC$. There exists a stable rank $r$  vector bundle $\cF$ on $\cC$  such that $\det \cF\cong \cL$ if and only if $r$ is coprime to $\deg \cL$. If that is the case, then such a vector bundle is unique up to isomorphism. 
\end{thm}
Now let $K:=\CC(\PP^1)$ and let $\cC\to K$ be the generic fiber of our elliptic  $K3$ surface. Then 
$\ell\in H^{1,1}_{\ZZ}(S)$ corresponds to a unique divisor class on $S$, and hence it determines a  line bundle $\cL$ on $\cC$. 
By the hypotheses of Theorem~\ref{vbk3ell},  $r$ is coprime to $\deg \cL$. Let $\cF$ be the (unique up to isomorphism) stable rank $r$ vector bundle on $\cC$ with determinant isomorphic to $\cL$. There exists a (non unique) coherent sheaf $\cE$ on $S$ which gives $\cF$ after base change. The restriction of $\cE$ to a generic elliptic fiber  is locally free  and  slope-stable. It follows that $\cE$ is locally free away from a finite union of elliptic fibers. 
Replacing $\cE$ by its double dual the restriction to a generic fiber does not change, and in addition $\cE$ is now locally free.  Since $c_1(\cE)=\ell+xf$ for some $x\in\ZZ$, we have proved Proposition~\ref{fibvettgen}.
\subsection{Elementary modifications along elliptic fibers}\label{subsec:gianicolense}
\setcounter{equation}{0}
Let $\cE$ be a vector bundle on $S$ satisfying the thesis of Proposition~\ref{fibvettgen}. Let  $C$ be an elliptic  fiber. Suppose that 
\begin{equation}\label{abici}
0\lra \cA\lra \cE_{|C}\lra \cB \lra 0,
\end{equation}
is an exact sequence where $\cB$ is torsion-free (recall  $\rho(S)=2$ and hence $C$ is irreducible). Let $\cE_1$ be the associated elementary modification of $\cE$, i.e.~the sheaf on $S$ fitting into the exact sequence 
\begin{equation}\label{traselm}
0\lra \cE_1\lra \cE\lra i_{*}\cB \lra 0,
\end{equation}
where $i\colon C\hra S$ is the inclusion map and $ \cE\lra i_{*}\cB$ is obtained by restricting to $C$ and composing with the map in~\eqref{abici}. Since $\cB$ is torsion-free, the sheaf $i_{*}\cB$ has depth $1$, and hence it has projective dimension $1$.  It follows that $\cE_1$ is locally-free. Restricting~\eqref{traselm} to $C$ one gets an exact sequence
\begin{equation}
0\lra \cB\lra \cE_{1|C}\lra \cA \lra 0.
\end{equation}
One reconstructs $\cE$ from the above exact sequence because we have
\begin{equation}
 0\lra \cE\otimes\cO_S(-C)\lra \cE_1\lra i_{*}\cA \lra 0.
\end{equation}
Notice that the restrictions of $\cE_1$ and $\cE$ to an elliptic fiber different from $C$ are isomorphic, and hence $\cE_1$ is $h$ slope-stable by Subsection~\ref{subsec:algardi}. Thus $[\cE_1]\in\cM_{\vv_1}(S,h)$ where $\ww_1:=v(\cE_1)$. Let  $\ww:=v(\cE)$. We have 
\begin{equation}
\ww_1=v(\cE_1)=v(\cE)-v(i_{*}\cB)=\ww-(0,r(\cB)f,\deg\cB).
\end{equation}
Hence
\begin{equation}\label{duequad}
\ww_1^2=\ww^2-2r\cdot r(\cB)\cdot\left(\mu(\cF_{|C})-\mu(\cB)\right).
\end{equation}
\begin{prp}\label{tondo}
Let $\cE$ be a vector bundle on $S$ satisfying the thesis of Proposition~\ref{fibvettgen}. There exists a vector bundle $\cF$ on $S$, isomorphic to $\cE$ away from a finite union of fibers, such that 
 the restriction of $\cF$ to \emph{every}  elliptic  fiber is stable. 
\end{prp}
\begin{proof}
Suppose that there exists an elliptic fiber $C$ such that $\cE_{|C}$ is not slope-stable.  Then $\cE_{|C}$ is slope-unstable because  $r$ is coprime to $q_S(f,\ell)$. Let~\eqref{abici} be a desemistabilizing quotient with $\cB$ torsion-free, and let $\cE_1$ be the elementary modification of $\cE$ defined by~\eqref{traselm}. 
 Then $v(\cE_1)^2<v(\cE)^2$ by~\eqref{duequad}. 
 If $\cE_1$ restricts to a non stable vector bundle on an elliptic  fiber  
we iterate this process. We get vector bundles $\cE=\cE_0$, $\cE_1,\ldots,\cE_m$, and we must stop after a finite number of steps because each $\cE_i$ is $h$ slope-stable by Subsection~\ref{subsec:algardi}, and $v(\cE_0)^2>v(\cE_1)^2>\ldots \ge -2$. The last vector bundle $\cF=\cE_m$  restricts to a stable vector bundle on every elliptic  fiber. 
\end{proof}
\subsection{Zero dimensional moduli spaces}\label{subsec:renatozero}
\setcounter{equation}{0}
\begin{prp}\label{enerzero}
There exists a vector bundle $\cF$ on $S$ such that 
\begin{equation*}
v(\cF)=(r,\ell+xf,t)
\end{equation*}
 for some $x,t\in\ZZ$, with the property that the restriction of $\cF$ to every elliptic fiber  is stable. Moreover, letting $\ww:=v(\cF)$, we have $\cM_{\ww}(S,h)=\{[\cF]\}$ and $\ww^2=-2$.
\end{prp}
\begin{proof}
Existence of $\cF$ has been proved in Proposition~\ref{tondo}. It remains to prove  that   $\cM_{\ww}(S,h)=\{[\cF]\}$ and $\ww^2=-2$. We claim that it suffices to prove that    if 
 $[\cG]\in\cM_{\ww}(S,h)$ with $\cG$ locally-free, then $\cG\cong\cF$. In fact this statement implies that $\ww^2=-2$ by Theorem~\ref{dimemme}, and then it follows also that every sheaf parametrized by $\cM_{\ww}(S,h)$ is locally-free (if 
 $[\cH]\in\cM_{\ww}(S,h)$  is not locally-free,  
  then $\cH$ is $h$ slope-stable by Subsection~\ref{subsec:algardi}, hence so is the double dual $\cH^{**}$, but  $v(\cH^{**})^2<v(\cH)^2=-2$, contradicting
 Theorem~\ref{dimemme}). Now suppose that  $[\cG]\in\cM_{\ww}(S,h)$ with $\cG$ locally-free. By Subsection~\ref{subsec:algardi} the restriction of $\cG$ to a generic  elliptic fiber is stable. Assume first that the restriction of $\cG$ to all elliptic fibers is stable. 
  By Atiyah's Theorem~\ref{fibell} the restrictions of $\cF$ and $\cG$ to each  elliptic fiber are isomorphic stable vector bundles. It follows that there exists a line-bundle $\cL_0$ on $\PP^1$ such that $\cG\cong\cF\otimes\pi^{*}\cL_0$. Since $v(\cG)=v(\cF)$ it follows that $\cL_0$ is trivial, i.e.~$\cG\cong\cF$. It remains to show the  that the restriction of $\cG$ to all fibers is stable. Suppose the contrary. By applying the procedure described above we arrive at a vector bundle $\cG_m$ such that $v(\cG_m)^2<v(\cG)^2$ and  the 
   restriction of $\cG_m$ to all fibers is stable. Arguing as above we get that there exists  a line-bundle $\cL_0$ on $\PP^1$ such that $\cG_m\cong\cF\otimes\pi^{*}\cL_0$. Hence 
\begin{equation}
v(\cG_m)^2=v(\cF\otimes\pi^{*}\cL_0)^2=v(\cF)^2=v(\cG)^2.
\end{equation}
This is a contradiction.
\end{proof}
At this point we have proved the statement of Theorem~\ref{vbk3ell} in the case $n(\vv)=0$. In fact it follows from Subsection~\ref{subsec:torcere} and Proposition~\ref{enerzero}.
\subsection{Moduli spaces of positive dimension}\label{subsec:positivo}
\setcounter{equation}{0}
We recall that $\dim\cM_{\vv}(S,h)=2n(\vv)$. We assume that $n(\vv)>0$. In order to simplify notation we let $n=n(\vv)$. The key step is to define a rational map from $T^{[n]}$ to   $\cM_{\vv}(S,h)$, where $T$ is a $K3$ surface.
Since the result is trivially true for $r=1$, we may assume that $r\ge 2$. 
Since $q_S(f,\ell)$ and $r$ are coprime (and  $r\ge 2$) there exist (unique) integers $r_0,d_0$ such that
\begin{equation}\label{mascomdiv}
q_S(f,\ell)\cdot r_0-r\cdot d_0=1,\quad 0<r_0<r.
\end{equation}
Let
\begin{equation}
\ww:=(r,\ell+n (r-r_0) f,s+n(q_S(f,\ell)-d_0)).
\end{equation}
Then $\ww^2=-2$, and hence by Subsection~\ref{subsec:renatozero} there exists an $h$ slope-stable vector bundle $\cF$ on $S$ (unique up to isomorphism) with $v(\cF)=\ww$. Moreover the restriction of $\cF$ to every elliptic fiber is stable by Proposition~\ref{enerzero}. Let 
$T:=\ov{\Pic}^{d_0}(S/\PP^1)$ be the relative moduli scheme parametrizing degree $d_0$ rank $1$ torsion-free sheaves on  fibers of $\pi\colon S\to\PP^1$.  Then $T$ is a $K3$ surface (it is smooth by deformation theory, and by deforming the elliptic fibration 
$\pi\colon S\to\PP^1$ to an elliptic $K3$ fibration $\pi_0\colon S_0\to\PP^1$  with a section, we see that it deforms to $S_0$).

We define a rational map 
\begin{equation}\label{aladin}
T^{[n]}  \dra  \cM_{\vv}(S,h) \\
\end{equation}
as follows. A generic point of $T^{[n]}$ corresponds to  an unordered $n$-tuple 
\begin{equation}\label{wolfie}
(C_1,L_1),\ldots,(C_n,L_n),
\end{equation}
 where 
$C_1,\ldots,C_n$ are distinct elliptic fibers, and $L_i$ is a line-bundle (unique up to isomorphism) of degree $d_0$ on $C_i$. For $k\in\{1,\ldots,n\}$ let $\cB_k$ be a stable rank $r_0$  vector bundle of degree $d_0$ on $C_k$ - such a vector bundle exists and is unique  up to isomorphism by~\eqref{mascomdiv} and Atiyah's Theorem. Since $\cF_{|C_k}$ and $\cB_k$ are stable, Equation~\eqref{mascomdiv} gives that $H^0(C_k,(\cF_{|C_k})^{\vee}\otimes \cB_k)=0$ and hence
\begin{equation*}
h^0(C_k,\cB_k^{\vee}\otimes (\cF_{|C_k}))=\chi(C_k,\cB_k^{\vee}\otimes (\cF_{|C_k}))=1.
\end{equation*}
By stability of $ \cF_{|C_k}$, a non zero map $\cB_k\to \cF_{|C_k}$ is an injection of vector bundles. It follows that we have an exact sequence of vector bundles
\begin{equation}\label{dritto}
0\lra \cB_k \lra \cF_{|C_k}\lra \cA_k\lra 0,
\end{equation}
unique up to isomorphism.  Let $\cE$ be the elementary modification of $\cF$ defined by the exact sequence
\begin{equation}\label{sambuco}
0\lra \cE \lra \cF\lra \bigoplus\limits_{k=1}^n i_{k,*}(\cA_k)\lra 0,
\end{equation}
where $i_k\colon C_k\to S$ is the inclusion map. Then $\cE$ is a vector bundle. Since the restrictions of $\cE$ and $\cF$ to an elliptic fiber different from one of $C_1,\ldots,C_n$ are isomorphic, $\cE$ is $h$ slope-stable. We have
\begin{equation*}
v( \cE)= v(\cF)-\sum\limits_{k=1}^n v(i_{k,*}(\cA_k))=v(\cF)-\sum\limits_{k=1}^n (0,(r-r_0)f,q_S(f,\ell)-d_0)=(r,\ell,s).
\end{equation*}
Hence $[\cE]\in\cM_{\vv}(S,h)$. We define the rational map in~\eqref{aladin} by sending to the point in $T^{[n]}$ corresponding to the unordered $n$-tuple in~\eqref{wolfie} to $[\cE]$. We claim that the map has degree $1$ onto its image. In fact suppose that 
$\cE$ fits into the exact sequence in~\eqref{sambuco}. Let $k\in\{1,\ldots,n\}$. Restricting~\eqref{sambuco} to $C_k$ we get an exact sequence
\begin{equation}\label{rovescio}
0\lra \cA_k \lra \cE_{|C_k}\lra \cB_k\lra 0.
\end{equation}
By~\eqref{mascomdiv} we have $\mu(\cF_{|C_k})>\mu(\cB_k)$, and hence~\eqref{rovescio} is a desemistabilizing squence of 
$\cF_{|C_k}$. In fact~\eqref{rovescio}  is the Harder-Narasimhan filtration of $\cF_{|C_k}$, because $\cB_k$ is stable by choice and  $\cA_k$ is stable by the exact sequence in~\eqref{dritto}, stability of $\cF_{|C_k}$, and Equation~\eqref{mascomdiv}. The conclusion is that $C_1,\ldots,C_n$ are the elliptic fibers with the property that the restriction of $\cE$ is not stable, and 
for each  $k\in\{1,\ldots,n\}$ the vector bundle $\cB_k$ is determined as the (unique) destabilizing locally-free quotient of 
$\cE_{|C_k}$ (and hence its determinant is uniquely determined). Hence  a generic point of the image of  $T^{[n]}$ comes form a unique point in $T^{[n]}$. Since $T^{[n]}$ is irreducible of dimension $2n$, the closure of the image of the map in~\eqref{aladin} is an irreducible component of $\cM_{\vv}(S,h)$, and hence also a connected component $\cM_{\vv}(S,h)_0$ because  $\cM_{\vv}(S,h)$ is smooth. Since $\cM_{\vv}(S,h)$ carries the Mukai-Tyurini holomorphic sympectic form, it follows that $\cM_{\vv}(S,h)_0$ is a HK variety birational to $T^{[n]}$, and hence also a deformation of $T^{[n]}$ by Huybrechts' fundamental result~\cite{pippo}.

In order to finish the proof of Theorem~\ref{vbk3ell} we must show that $\cM_{\vv}(S,h)_0=\cM_{\vv}(S,h)$. By the discussion above it suffices to prove that if $[\cE]\in\cM_{\vv}(S,h)$ is generic then the following hold:
\begin{enumerate}
\item
$\cE$ is locally-free,
\item
the number of elliptic fibers with the property that the restriction of $\cE$ is not stable is $n$, 
\item
and, letting $C_1,\ldots,C_n$ be the elliptic fibers of Item~(2), 
 the Harder Narasimhan filtration of $\cE_{|C_k}$ (for  $k\in\{1,\ldots,n\}$) is as in~\eqref{rovescio}, where $r_0:=r(\cB_k)$ and $d_0:=\deg(\cB_k)$ satisfy Equation~\eqref{mascomdiv}.
\end{enumerate}
The statements above follow from dimension counts. In order to prove that Item~(1) holds (of course this is where the hypothesis $r\ge 2$ is necessary), assume the contrary. Then there exists an open non empty subset of $\cM_{\vv}(S,h)$ parametrizing sheaves $\cE$ such that $\cE^{\vee\vee}/\cE$ has a fixed length $d>0$. Counting parameters and invoking the main result of~\cite{ellehn-irr-quot}, which gives $d(r+1)$ as the dimension of the parameter space for quotients of $\cE^{\vee\vee}$ of length $d$, we get that
\begin{multline*}
2n=\dim\cM_{\vv}(S,h)=v(\cE^{\vee\vee})^2+2+d(r+1)= \\ 
=\la(r,\ell,s+d),(r,\ell,s+d)\ra+2+d(r+1)=2n-(r-1)d.
\end{multline*}
This is a contradiction because $r\ge 2$ (and $d>0$). In order to prove that Items~(2) and~(3) hold, one proceeds by induction on $n$. The case $n=0$ holds by Subsection~\ref{subsec:renatozero}. Let $n>0$ and let $[\cE]$ be generic in an irreducible component of $\cM_{\vv}(S,h)$. It suffices  to show that if $C$ is an elliptic fiber such that $\cE_{|C}$ is not stable, then its Harder-Narasimhan filtration is as in Item~(3) because then also Item~(2) will follow. We assume that this is not the case, and we get a contradiction by counting parameters - see Step 4.

\section{Moduli of  sheaves on $K3$ surfaces}
\subsection{The result}
\setcounter{equation}{0}
Building  on Theorem~\ref{vbk3ell} we prove  Theorem~\ref{modonk3} under the hypothesis that  $r+\ell$  is indivisible. We spell out the result   in the language of Mukai vectors.
\begin{thm}[O'Grady~\cite{ogwt2}]\label{casoprim}
Let $S$ be a projective $K3$ surface. Let $\vv=(r,\ell,s)\in H(S;\ZZ)$.
Suppose that $\vv^2\ge -2$ and that $r+\ell$ is primitive.  If $h$ is an $a(\vv)$-generic polarization of $S$ then  $\cM_{\vv}(S,h)$ is non empty, irreducible, smooth of dimension 
$\vv^2+2$. If $n(\vv)>0$ (see~\eqref{enea}) then  $\cM_{\vv}(S,h)$  is a HK variety of Type $K3^{[n(\vv)]}$.
\end{thm}
\begin{rmk}\label{rmk:puzzone}
In~\cite{ogwt2} the hypothesis is "$\ell$  primitive". A simple argument (similar to what is done in Subsection~\ref{subsec:multacca})  shows that the special case treated in~\cite{ogwt2} implies the validity of the result under the hypothesis that $r+\ell$ is primitive. 
\end{rmk}
If $r=1$ then Theorem~\ref{casoprim} is trivially true because $\cM_{\vv}(S,h)$ \emph{is} isomorphic to $S^{[n(\vv)]}$. For this reason we assume  throughout the present section that $r\ge 2$.
\subsection{We may assume that $\ell$ is a multiple of $h$.}\label{subsec:multacca}
\setcounter{equation}{0}
We prove that it suffices to prove Theorem~\ref{casoprim} under the hypothesis that 
\begin{equation}\label{ellacca}
\vv=(r,xh,s),\quad \gcd\{r,x\}=1.
\end{equation}
First we prove a key consequence of Proposition~\ref{campol} (see also Subsection~\ref{subsec:algardi}). 
\begin{lmm}\label{solostab}
Let $S$ be a projective $K3$ surface. Let $\vv=(r,\ell,s)\in H(S;\ZZ)$, and 
suppose that $r+\ell$ is primitive.  If  $h$ is an $a(\vv)$-generic polarization of $S$ then every sheaf parametrized by  $\cM_{\vv}(S,h)$ is slope stable.
\end{lmm}
\begin{proof}
Suppose that $[\cF]\in \cM_{\vv}(S,h)$ and $\cF$ is not  $h$ slope-stable. Then $\cF$ is strictly 
$h$ slope-semistable because it is $h$ GM semistable. By Proposition~\ref{campol} it follows that there exists a subsheaf $\cE\subset\cF$ such that $0<r(\cE)<r$ and 
$$r c_1(\cE)=r(\cE)\ell.$$
 This contradicts the hypothesis  that $r+\ell$ is primitive (i.e.~$r$ is coprime to the divisibility of $\ell$).
\end{proof}
Next, given $N\in\NN$ let 
$$\vv(N):=\vv\smile \exp{Nh}=(r,\ell+r N h,*).$$
Tensorization by $\cO_S(NH)$ defines an isomorphism
$$\cM_{\vv}(S,h)\overset{\sim}{\lra} \cM_{\vv(N)}(S,h).$$
Moreover if $N$ is very large then $\ell+r N h$ is ample and it belongs to the same open $a(\vv)$-chamber as $h$ does (notice that $a(\vv)=a(\vv(N))$). Arguing as in Subsection~\ref{subsec:algardi} one shows that every sheaf in $\cM_{\vv(N)}(S,h)$ is $h$ slope stable. By Proposition~\ref{campol} it follows that 
$\cM_{\vv(N)}(S,h)\cong \cM_{\vv(N)}(S,\ell+r N h)$ or, more precisely, 
the two moduli spaces  parametrize the same isomorphism classes of sheaves. Let $\ell+rN=x \ov{h}$ where $x>0$ and $\ov{h}$ is primitive (and ample of course). Since $r+\ell$ is primitive we have $\gcd\{r,x\}=1$. Renaming $\ov{h}$ by $h$ we get that it suffices to prove Theorem~\ref{casoprim} for $\vv$ as in~\eqref{ellacca}.
\subsection{Elliptic Noether-Lefschetz loci}
\setcounter{equation}{0}
We will be use the following  result (the proof is elementary, see Lemma~4.3 in~\cite{ogmodvb}).
\begin{lmm}\label{nocamere}
Let $(\Lambda,q)$ be a non degenerate rank $2$ lattice which represents $0$, and hence $\disc(\Lambda)=-d^2$ where $d$ is a strictly positive integer. Let $\alpha\in \Lambda$ be primitive isotropic, and complete it to a basis $\{\alpha,\beta\}$ such that $q(\beta)\ge 0$. If $\gamma\in\Lambda$ has  negative square (i.e.~$q(\gamma)<0$) then
\begin{equation}\label{menoenne}
 q(\gamma)\le -\frac{2d}{1+q(\beta)}.
\end{equation}
\end{lmm}
Let $\cK_e$ be the moduli space of polarized $K3$ surfaces of degree $e$.  
\begin{dfn}\label{luogoell}
For $d>0$ let $\cN_e(d)\subset \cK_e$ be the set of $[(S,h)]$ such that $\NS(S)$ contains a class $f$ with the property that 
 $\la h,f\ra$ is saturated in $\NS(S)$ and
\begin{equation}\label{prododi}
q_{S}(h,f)=d,\quad q_{S}(f,f)=0. 
\end{equation}
\end{dfn}
The subset $\cN_e(d)$ is a Noether-Lefschetz locus and hence it is closed of pure codimension $1$. 
\begin{prp}\label{unafibrazione}
Suppose  that 
 \begin{equation}\label{chiarello}
d>(e+1),\quad e\notdivides d.
\end{equation}
If  $[(S,h)]\in\cN_e(d)$ is  generic  there is one and only one elliptic fibration 
\begin{equation}\label{congiunti}
\pi\colon S\to\PP^1
\end{equation}
such that, letting $f:=\pi_t^{*}(c_1(\cO_{\PP^1}(1)))$, the lattice $\la h,f\ra$ is saturated  in $\NS(S)$ and~\eqref{prododi} holds. 
\end{prp}
\begin{proof}
Let $S$ be a $K3$ surface   such that 
\begin{equation}
\NS(S)=\la h,f\ra,\quad q_{S}(h,f)=d,\quad q_{S}(f,f)=0. 
\end{equation}
There are no  $\xi\in \NS(S)$ such that $q_{S}(\xi)=-2$ by  the inequality in~\eqref{chiarello} and Lemma~\ref{nocamere} - hence there are no smooth rational curves on $S$.
It follows that the  ample cone of $S$ is equal to the intersection of $\NS(S)$ and the positive cone. Hence, changing sign to $h$ and $f$ if necessary, we may assume that $h$  is ample. 

A straightforward computation shows  that there are exactly two primitive nef isotropic classes, namely  $f$ and $\alpha:=\frac{1}{\gcd\{d,e\}}(2d h-e f)$. 
  By our \lq\lq non divisibility\rq\rq\ hypothesis in~\eqref{chiarello}, we get that  $q_S(\alpha,h)=\frac{de}{\gcd\{d,e\}}$ is not equal to $d$. Hence $f$ is the unique primitive nef isotropic  class  such that $q_S(h, f)=d$. Let $F$ be a divisor represented by $f$. The invertible sheaf  $\cO_S(F)$ is globally generated  because there are no smooth rational curve on $S$, and since $q_S(f,f)=0$, the map $S\lra |F|^{\vee}$ is an elliptic fibration.

A polarized $K3$ surface $(S,h)$ as above is represented by a point of $\cN_e(d)$, and the set of points of $\cN_e(d)$ representing such surfaces is dense in $\cN_e(d)$. The proposition follows.
\end{proof}
The result below allows us to prove Theorem~\ref{casoprim} by appealing to Theorem~\ref{vbk3ell}.
\begin{prp}\label{adatto}
Suppose  that 
 \begin{equation}\label{chiarello}
d>\frac{(e+1)}{8} r^2(\vv^2+2r^2),\quad e\notdivides d.
\end{equation}
Let  $[(S,h)]\in\cN_e(d)$ be \emph{very} generic, and hence we may assume that the thesis of Proposition~\ref{unafibrazione} holds (recall that $r\ge 2$ and $\vv^2\ge -2$) and $\NS(S)=\la h,f\ra$.   Then $h$ is $a(\vv)$-generic with respect to the elliptic fibration in~\eqref{congiunti} (recall that $\vv=(r,xh,s)$).  
\end{prp}
\begin{proof}
Using Lemma~\ref{nocamere} one checks that there is single $a(\vv)$-chamber.
\end{proof}
\subsection{Proof of Theorem~\ref{casoprim}}\label{subsec:relativo}
\setcounter{equation}{0}
By Subsection~\ref{subsec:multacca} it suffices to prove Theorem~\ref{casoprim} for $\vv$ given by~\eqref{ellacca}. Let $e:=q_S(h, h)$, and let
 $\cX\to T_e$ be a complete family of polarized $K3$ surfaces  of degree $e$. Since $\cK_e$  is irreducible, we may assume that $T_e$ is irreducible. If $t\in T_e$ we let $(S(t),h(t))$ be the polarized $K3$ surface corresponding to $t$, and $\vv(t):=(r,xh(t),s)$. By Maruyama there exists a relative moduli space over $T_e$, i.e.~a map of schemes
$$\rho\colon\cM\lra T_e$$
such that for $t\in T_e$ the fiber $\rho^{-1}(t)$ is isomorphic to $\cM_{\vv(t)}(S(t),h(t))$. Moreover the map $\rho$ is projective, in particular it is proper. 

Let $T^0_e\subset T_e$ be the subset of points $t$ such that  $h(t)$ is $a(\vv(t))$-generic. Then $T^0_e$ is open dense, and by hypothesis there exists $t_0\in T^0_e$ such that $(S({t_0}),h({t_0}))\cong (S,h)$. Let $\cM^0:=\rho^{-1}(T^0_e)$. The restriction of $\rho$ defines a projective map 
$\cM^0\to T^0_e$.
\begin{clm}\label{reliscio}
The map $\cM^0\to T^0_e$ is smooth.
\end{clm}
\begin{proof}
A point of $\cM^0$ represents an $h(t)$ semistable sheaf $\cF$ on $S(t)$, where $t\in T^0_e$. The polarization is $a(\vv(t))$-generic by definition of $T^0_e$, and hence $\cF$ is slope stable. By Serre duality 
$$\Ext^2(\cF,\cF)^0\cong (\Hom(\cF,\cF)^0)^{\vee},$$
where $^0$ means "traceless". Since $\cF$ is (slope) stable it has no  nonzero traceless endomorphisms. By~\cite{man-iac-pairs} it follows that the map $\cM^0\to T^0_e$ is smooth.
\end{proof}
Let $d$ be as in Proposition~\ref{adatto}. Let $t\in T_e$ be such that  $[(S(t),h(t)]$ is very generic point of $\cN_e(d)$. By Proposition~\ref{adatto}  every sheaf parametrized by $\cM_{\vv(t)}(S(t),h(t))$ is slope stable (see Subsection~\ref{subsec:algardi}) and hence $t\in T_e^0$. By Theorem~\ref{vbk3ell} the moduli space $\cM_{\vv(t)}(S(t),h(t))$ satisfies the thesis of Theorem~\ref{casoprim}, and hence so does $\cM_{\vv}(S,h)$ by Claim~\ref{reliscio}.
\qed
\section{Proof of Theorem~\ref{unicita}}
\subsection{Outline}
\setcounter{equation}{0}
The first subsection recalls the definition and some properties of semi homogeneous vector bundle on an abelian variety. Simple semi homogeneous vector bundles on abelian varieties of dimension greater than $2$ are similar to simple (i.e.~stable) vector bundles on elliptic curves in the following sense: they have no non trivial deformations which keep the determinant fixed. Notice that there exist simple vector bundles on abelian varieties of dimension greater than $2$ which have deformation spaces of arbitrarily large dimension (even if we fix the isomorphism class of the determinant).

In the second subsection we establish the connection between semi homogeneous vector bundles and modular sheaves on Lagrangian HK's. The main point is that if a modular sheaf restricts to a stable vector bundle on a generic Lagrangian fiber then the restriction is a semi homogeenous vector bundle. In particular under this hypothesis one gets strong constraints for the rank of a modular sheaf.

Subsection~\ref{subsec:lagdiv} introduces  Lagrangian Noether Lefschetz divisors in moduli spaces of polarized HK's of Type $K3^{[2]}$. 

Subsection~\ref{subsec:evadere} gives a (weak) analogue of the results in  Subsection~\ref{subsec:algardi}.

Our next task is to produce stable modular vector bundles with the prescribed rank, $c_1$ and $c_2$. Subsection~\ref{subsec:fallito} explains why one encounters problems in trying to extend the construction of Subsection~\ref{subsec:innocenzox}.  In Subsection~\ref{subsec:dibase} we give a construction of modular vector bundles with the correct rank, $c_1$ and $c_2$ on a Hilbert square of a $K3$ surface $S$.  In Subsection~\ref{subsec:fossanova} we specialize $S$ to be elliptic so that the Hilbert square is Lagrangian.

In Subsection~\ref{subsec:vblag} we show that there are irreducible components of Lagrangian Noether-Lefschetz whose generic point parametrizes polarized HK's $(X,h)$ carrying a slope stable vector bundle  (with the correct rank, $c_1$ and $c_2$) whose restriction to Lagrangian fibers is slope stable with the possible exception of a finite set of fibers - this reproduces the behaviour of vector bundle on elliptic $K3$ surfaces which have no infinitesimal deformations. 

The last subsection wraps it all up.
\subsection{Semi homogeneous vector bundles on abelian varieties}\label{subsec:semihom}
\setcounter{equation}{0}
  Let $A$ be an abelian  variety, and   let $T_a\colon A\to A$ be the translation by $a\in A$.  
\begin{dfn}
  A vector bundle  $\cF$ on $A$ is  \emph{semi homogeneous} if, for every $a\in A$, there exists an invertible sheaf $\xi$ on $A$ such that $T_a^{*}\cF\cong \cF\otimes\xi$. 
\end{dfn}
\begin{rmk}
If $\cF$ is a semi homogeneous vector bundle on an abelian variety then $\Delta(\cF)=0$. In fact $End(\cF)$ is a homogeneous vector bundle (Theorem 5.8 in~\cite{muksemi}) and hence is topologically trivial  (Theorem 4.17 in~\cite{muksemi}). 
Thus  $\Delta(\cF)=c_2(End\cF)=0$.
\end{rmk}
The following result follows from Theorem 5.8 in~\cite{muksemi} and the Kobayashi-Hitchin correspondence proved by Uhlenbeck-Yau~\cite{uhlyau}, see~\cite{ogmodvb}.
\begin{prp}\label{discperp}
Let $(A,\theta)$ be a polarized abelian variety  of dimension $n$, and let   $\cF$ be a $\theta$ slope-stable vector bundle on $A$.  If
\begin{equation}\label{delper}
\int_A \Delta(\cF)\smile \theta^{n-2}=0
\end{equation}
(the condition is to be understood to be empty if $n=1$) then $\cF$ is simple semi homogeneous. 
\end{prp}
As shown by Mukai there are strong restrictions on the rank and $c_1$ 
of a simple semi homogeneous vector bundle.  Below is an extension of  results  
of Mukai (see Theorem 7.11 and Remark 7.13 in~\cite{muksemi})  proved in~\cite{ogmodvb}.  
\begin{prp}[Mukai~\cite{muksemi}]\label{potenza}
Let $(A,\theta)$ be a  polarized abelian variety of dimension $n$. Suppose that the elementary divisors of $\theta$ are $(1,\ldots,1,d_1,d_2)$ where $d_1$ divides $d_2$.  Let $\cF$ be a simple semi homogeneous vector bundle  on $A$ such that $c_1(\cF)=a\theta$. Then there exists a positive integer $r_0$    such that, letting $g_i:=\gcd\{r_0,d_i\}$ we have   
\begin{equation}\label{controllo}
\gcd\{r(\cF),a\}=\frac{r_0^{n-1}}{g_1\cdot g_2},\quad  r(\cF)=\frac{r_0^{n}}{g_1\cdot g_2}.
\end{equation}
\end{prp}
\begin{prp}[Mukai, Proposition 7.1 in~\cite{muksemi}]\label{quadrato}
Let $\cF$ be a simple semi homogeneous vector bundle on an abelian variety. Then the set of line bundle $\xi$ such that
$ \cF\cong \cF\otimes\xi$ has cardinality $r(\cF)^2$. 
\end{prp}
The above result points out a difference between simple semi homogeneous vector bundle on elliptic curves  (i.e.~stable vector bundles) and on higher dimensional abelian varieties. 
 If $A$ is an elliptic curve then a stable vector bundle is isomorphic  
to any other stable vector bundle with the same rank and determinant. 
On the other hand let $(A,\theta)$ be a polarized abelian variety such that $\dim A\ge 2$, and consider  a   moduli space $\cM$ of slope stable vector bundles on $A$ with fixed rank $r$, determinant and $c_2$, under the hypothesis that $2r c_2-(r-1)c_1^2=0$, i.e.~the bundles are semi homogeneous (see  Proposition~\ref{discperp}). Then $\cM$, if not empty, is zero dimensional  (see Proposition 5.9 in~\cite{muksemi}) but    is not  a singleton. In fact suppose that $[\cF]\in\cM$. If $\xi$ is an $r$-torsion line bundle  then  $\cF\otimes\xi$ has the same rank, determinant and $c_2$ as $\cF$, but  for most $\xi$  it is not isomorphic to $\cF$  by Proposition~\ref{quadrato}.

Because of this fact the analogue of Proposition~\ref{enerzero} for $X$ of Type $K3^{[2]}$ requires an additional argument.

\subsection{Modular sheaves on Lagrangian fibrations and semi homogeneous vector bundles}\label{subsec:semhomlag}
\setcounter{equation}{0}
Let  $\pi\colon X\to\PP^n$ be  a Lagrangian fibration on a HK $X$. We keep the notation introduced in Subsection~\ref{subsec:pazzia}.   In particular if $X_t$ is a smooth fiber then $\theta_t\in H^{1,1}_{\ZZ}(X_t)$ is the ample class of Remark~\ref{eccoteta}, and 
$f:=c_1(\pi^{*}\cO_{\PP^n}(1))$.
As is well-known $q_X(f,f)=0$.
\begin{lmm}\label{zerene}
Let  $\pi\colon X\to\PP^n$ be a  Lagrangian fibration of a HK manifold of dimension $2n$. Suppose that $\cF$ is a modular torsion free sheaf on $X$. Let $t\in\PP^n$ be a general point, and let $\cF_t:=\cF_{|X_t}$ be the restriction of $\cF$ to $X_t$. Then 
\begin{equation}\label{intzero}
\int_{X_t}\Delta(\cF_t)\smile \theta_t^{n-2}=0.
\end{equation}
\end{lmm}
\begin{proof}
There exists $\rho>0$ such that $\omega_{|X_t}=\rho\cdot \theta_t$. Since $t\in\PP^n$ is a generic point, we have $\Delta(\cF_t)=\Delta(\cF)_{|X_t}$. Moreover $f^n$ is the Poincar\'e dual of $X_t$.  Hence
\begin{equation}
\rho^{n-2}\int_{X_t}\Delta(\cF_t)\smile \theta_t^{n-2}=\int_{X}\Delta(\cF)\smile \omega^{n-2}\smile f^n.
\end{equation}
The integral on the right vanishes by Remark~\ref{deltanti} and the equality
 $q_X(f)=0$.
\end{proof}
The result below, which follows at once from Lemma~\ref{zerene} and Proposition~\ref{discperp}, gives the connection between modular sheaves on Lagrangian HK varieties and semi homogeneous vector bundles. 
\begin{prp}\label{resemi}
Let  $\pi\colon X\to\PP^n$ be a   Lagrangian fibration of a HK manifold of dimension $2n$. Let $\cF$ be a modular torsion free sheaf on $X$. 
Suppose that $t\in\PP^n$ is a regular value of $\pi$, that $\cF$ is locally-free in a neighborhood of $X_t$, and that $\cF_t$ is slope-stable. Then $\cF_t$ is a semi homogeneous vector bundle.
\end{prp}
The following result gives a strong restriction on the rank of a modular sheaf on Lagrangian HK varieties of  Type $K3^{[n]}$ or $\Kum_n$ for $n\ge 2$, or of Type OG6 under the hypothesis that the restriction to a generic fiber is a slope stable vector bundle.
\begin{crl}\label{resemibis}
Let $X$ be a HK of Type $K3^{[n]}$, $\Kum_n$ or OG6. Let $\cF$ be a modular torsion free sheaf on $X$. 
Suppose that $t\in\PP^n$ is a regular value of $\pi$, that $\cF$ is locally-free in a neighborhood of $X_t$, and that $\cF_t$ is slope-stable. Then there exist  positive integers $r_0,d$, with $d$ dividing $c_X$,   such that   
$ r(\cF)= \frac{r_0^{n}}{d}$. 
\end{crl}
\begin{proof}
If $X$ is of Type $K3^{[n]}$ then $c_X=1$ and $\theta_t$ is a principal polarization, see~\cite{wieneck1}. If $X$ is of Type $\Kum_n$ or OG6 then $c_X=n+1$ and $\theta_t$ is a polarization with elementary divisors $(1,\ldots,1,d_1,d_2)$ where $d_1\cdot d_2$ divides $n+1$ see~\cite{wieneck2} for $\Kum_n$ and~\cite{monrap} for OG6. 
Hence the result  follows from  Proposition~\ref{resemi} and Proposition~\ref{potenza}.
\end{proof}
We end the subsection with a result which is related to Corollary~\ref{resemibis} although there is no  lagrangain fibration in the hypothesis. For the proof see Proposition~2.3 in~\cite{ogmodvb}.
\begin{prp}\label{restrango}
Let $X$ be a HK fourfold of Type $K3^{[2]}$ or $\Kum_2$. Let $\cF$ be a modular torsion-free sheaf on $X$. Let $m$ be a generator of the ideal 
\begin{equation*}
\{q_X(c_1(\cF),\alpha)\mid \alpha\in H^2(X;\ZZ)\}.
\end{equation*}
 Then  $r(\cF)$ divides $m^2$ if $X$ is of Type $K3^{[2]}$, and it divides $3m^2$ if $X$ is of Type $\Kum_2$. 
\end{prp}
\subsection{Lagrangian Noether-Lefschetz divisors in $\cK^i_e$}\label{subsec:lagdiv}
\setcounter{equation}{0}
Recall that $\cK^i_e$ is the moduli space of polarized HK's of Type $K3^{[2]}$ with polarization of BBF square $e$ and divisibility given by $i$ (which is either $1$ or $2$) - see Subsection~\ref{risprin}. 
\begin{dfn}\label{ennelagr}
For $d$ a strictly positive integer let $\cN_e^i(d)\subset \cK_e^i$ be the closure of the set of points $[(X,h)]$  such that $H^{1,1}_{\ZZ}(X)$ 
contains a  saturated rank $2$ sublattice generated by $h,f$, where 
\begin{equation}\label{disez}
q_{X}(h, f)=d,\quad q_{X}(f,f)=0.
\end{equation}
\end{dfn}
Notice that $\cN_e^i(d)\subset \cK_e^i$ is a Noether-Lefschetz divisor analogous to $\cN_e(d)\subset T_e$, see  Definition~\ref{luogoell}. 

The next result is the analogue of Proposition~\ref{unafibrazione}, for the proof see Proposition B.2 in~\cite{ogmodvb}.
\begin{prp}\label{unicafibr}
 Keeping notation as above, suppose in addition that $d$ is even if $i=2$, and that 
 \begin{equation}\label{chiarello}
d>10(e+1),\quad e\notdivides 2d.
\end{equation}
Then $\cN_e^i(d)$ is closed of pure codimension $1$ (in particular non empty), and  if 
$[(X,h)]\in\cN_e^i(d)$ is  generic  there is one and only one Lagrangian fibration $\pi\colon X\to\PP^2$  such that, letting  $f:=c_1(\pi^{*}\cO_{\PP^2}(1))$,   the equalities in~\eqref{disez} hold and moreover the sublattice  
$\la h,f\ra\subset H^{1,1}_{\ZZ}(X)$ is saturated. 
\end{prp}
Let hypotheses be as in Proposition~\ref{unicafibr} and let $[(X,h)]\in\cN_e^i(d)$ be a  generic point. Let $\pi\colon X\to\PP^2$ be the unique Lagrangian fibration as in Proposition~\ref{unicafibr}. Let $U\subset\PP^2$ be the open dense set of regular values of $\pi$  and let $X(U):=\pi^{-1}(U)$. Let 
$\Pic^0(X(U)/U)$  be the relative Picard scheme, and for $t\in U$ let  $A_t$ be 
 the fiber of $\Pic^0(X_U/U)\to U$ over  $t$. Then $A_t$  is an abelian surface and
the fundamental group $\pi_1(U,t)$ acts by monodromy on the subgroup $A_{t,tors}$ of torsion points. The result below will be used later on (for the full proof see Corollary B.5 in~\cite{ogmodvb}). It is motivated by Proposition~\ref{quadrato} - recall that in Theorem~\ref{unicita} the bundles have rank $r_0^2$.
\begin{prp}\label{modinv}
Keep hypotheses as above, and suppose that 
 $V\subset A_t[r_0^2]$ is a coset (of a subgroup) of cardinality $r_0^4$ invariant under the action of monodromy. 
 Then $V= A_t[r_0]$.
\end{prp}
\begin{proof}[Idea of proof]
The fibers of $\pi\colon X\to\PP^2$ are integral, hence $\Pic^0(X/\PP^2)\to\PP^2$ is defined (independently of the choice of a polarization). 
One describes  $\Pic^0(X/\PP^2)\to\PP^2$ as follows.  Following Markman (Subsection~4.1 in~\cite{markman-lagr}) there is a generic polarized $K3$ surface $S$ of degree $2$ associated to $X$.  Let $f\colon S\to\PP^2$ be the corresponding double cover. We let $B\subset \PP^2$ be the branch divisor, a smooth generic sextic curve. Let $\cJ(S)\to(\PP^2)^{\vee}$ be the relative Jacobian with fiber  
$\Pic^0(f^{-1}(L))$ over a line $L$ - this is a HK of Type $K3^{[2]}$,  a moduli space of pure torsion sheaves on $S$. 
Let $\cJ_0(S)\subset \cJ(S)$ be the open dense subset of   smooth points  of the map 
$\cJ(S)\to(\PP^2)^{\vee}$ (i.e.~smooth points of $\cJ(S)$ with surjective differential). 
Using results of Markman, one shows that $\Pic^0(X/\PP^2)\to\PP^2$ is 
 isomorphic to $\cJ(S)_0\to(\PP^2)^{\vee}$, for a certain identification $\PP^2\overset{\sim}{\lra} (\PP^2)^{\vee}$. Under this identification $t\in\PP^2$ corresponds to a line $R\in (\PP^2)^{\vee}$
 transverse to $B$, and the corresponding Lagrangian fiber $A_t$   is the Jacobian of the  double cover of $R$ ramified over $R\cap B$.  Hence the  monodromy action on $H^1(A_t;\ZZ)$ is the full symplectic group, where the symplectic form is the one given by the prinicpal polarization on the Jacobain of the curve $f^{-1}(R)$. The result follows from this and a simple argument.
\end{proof}
\subsection{Stable vector bundles  on Lagrangian HK's and their restrictions to a generic fiber}\label{subsec:evadere}
\setcounter{equation}{0}
The main result of the present subsection, given below, is a partial analogue of the results in Subsection~\ref{subsec:algardi}.  In particular it gets around the difficulty that was pointed out at the end of Subsection~\ref{subsec:semihom}, namely that   moduli spaces of slope stable semi homogeneous vector bundles  with fixed rank, determinant and $c_2$ on abelian varieties of dimension at least $2$ are not   singletons.
\begin{prp}\label{propriostab}
Let $a_0,d$ be  positive integers and $i\in\{1,2\}$. Suppose that $e\not\divides 2d$, that $d$ is even if $i=2$, and that 
\begin{equation}\label{golfangora}
d>\max\left\{\frac{1}{2}a_0(e+1), 10(e+1)\right\}.
\end{equation}
If $[(X,h)]\in\cN_e^i(d)^0$ is generic the following hold:  
\begin{enumerate}
\item
Let  $\cE$ be an $h$ slope-stable vector bundle on $X$ such that 
\begin{enumerate}
\item[(a)]
$a(\cE)\le a_0$, where $a(\cE)$ is as in Definition~\ref{adieffe}, 
\item[(b)]
there exists an integer $m$ such that $r(\cE)=(mi)^2$, $c_1(\cE)=m h$, and $\gcd\{mi,\frac{d}{i}\}= 1$.
\end{enumerate}
Then
 the restriction of $\cE$ to a generic fiber of  the associated Lagrangian fibration $\pi\colon X\to\PP^2$ is slope-stable.
\item
If $\cE,\cE'$ are  $h$ slope-stable vector bundles on $X$ such that Items~(a) and (b)  hold for $\cE$ and $\cE'$, then for a generic $t\in\PP^2$ the restrictions of $\cE$ and $\cE'$ to $X_t$ are isomorphic. 
\end{enumerate}
\end{prp}
The proof of Item~(1) is a straightforward computation, see Proposition~4.1 in~\cite{ogmodvb}. Here we give the proof of Item~(2), which involves Proposition~\ref{quadrato} and Proposition~\ref{modinv}.

\begin{proof}[Proof of Item~(2) of Proposition~\ref{propriostab}]
For $t\in\PP^2$ let $\cF_t:=\cF_{|X_t}$, $\cG_t:=\cG_{|X_t}$.
By Item~(1) of Proposition~\ref{propriostab} there exists an open dense $U\subset\PP^2$ such that for $t\in U$ the vector bundles  $\cF_t$ and $\cG_t$ are 
both slope-stable. We may assume that $X_t$ is smooth for every $t\in U$. By Proposition~\ref{resemi} it follows that $\cF_t$ and $\cG_t$ are simple semi-homogeneous vector bundles.  Let  $t\in U$. By Theorem~7.11 in~\cite{muksemi} the set
\begin{equation*}
V_t:=\{[\xi]\in X_t^{\vee} \mid \cF_t\cong \cG_t\otimes\xi\}
\end{equation*}
is not empty, and hence it has cardinality $r(\cG)^2$ by Proposition~\ref{quadrato}. Clearly $V_t$ is invariant under the monodromy action of $\pi_1(U,t)$. 
Now notice that $V_t\subset A_t[(mi)^2]$ because $\cF_t$ and $\cG_t$ have rank $(mi)^2$ and isomorphic determinants. 
Hence  by Proposition~\ref{modinv} we have $V_t=A[mi]$. Thus $0\in V_t$, and therefore $\cF_t\cong \cG_t$. 
\end{proof}

\subsection{Search for an analogue of Subsection~\ref{subsec:innocenzox}}\label{subsec:fallito}
\setcounter{equation}{0}
We describe a failed attempt to produce an analogue of Subsection~\ref{subsec:innocenzox}. In fact this is the stage in which we find the greatest differences between dimension $2$ and higher dimension. Let $\pi\colon X\to\PP^n$ be a Lagrangian fibration of a HK variety $X$ of dimension $2n\ge 4$. Say we wish to construct a modular vector bundle  $\cE$ on $X$ such that 
\begin{enumerate}
\item 
$H^2(X,End_0(\cE))=0$, so that $\cE$ extends to a vector bundle on every small deformation of $X$ keeping $c_1(\cE)$ of type $(1,1)$.
\item 
For a generic $t\in\PP^n$ the restriction $\cE_t:=\cE_{|X_t}$ is stable, so that $\cE$ itself is $h$ slope stable for an $a(\cE)$-suitable polarization $h$. 
\end{enumerate}
One might try to imitate what was done in Subsection~\ref{subsec:innocenzox}, i.e.~consider the generic fiber $\cX\to\CC(\PP^n)$. This is an abelian variety over $\CC(\PP^n)$, more precisely  a torsor over an abelian variety over $\CC(\PP^n)$. Hence the definition of semi homogeneous vector bundle on $\cX$  makes sense (we formulated the notion for abelian varieties in the "strict" sense, but it clearly makes sense for a torsor over an abelian variety). Suppose that $\cF$ is such a vector bundle on $\cX$.    There are many choices of a torsion free sheaf $\cE$ which restricts to $\cF$ on the generic fiber of $\pi\colon X\to\PP^n$. In choosing  such an $\cE$ we ask that Items(1) and~(2) above hold and moreover that $\cE$ is modular. 
Regarding Item~(1) we notice that  we are better off if $\cE$  restricts to  a simple sheaf on $X_t$ for $t$ outside a set of codimension at least $2$ in $\PP^n$. In fact suppose that this is not the case, and let $D\subset\PP^n$ be a prime divisor with the property that $H^0(X_t,\End_0(\cE_t))\not=0$ for all $t\in D$ (this is almost a strengthening of Item~(2) above). We claim that then also $H^1(X_t,\End_0(\cE_t))\not=0$ for all $t\in D$. In fact if $H^1(X_t,\End_0(\cE_t))=0$ for a generic $t\in D$ then
$H^1(X_t,\End_0(\cE_t))=0$ for  $t\in U$ where $U\subset \PP^n$ is an open subset intersecting $D$, and then it follows that $H^0(X_t,\End_0(\cE_t))=0$ for a generic  $t\in D$ by the properties of cohomology and base change, see Corollary 2, pp.~50-51 in~\cite{mumabvar}. 
 Thus  $H^1(X_t,\End_0(\cE_t))\not=0$ for all $t\in D$. It follows that $R^1\pi_{*}\End_0(\cE)$ is non zero (and supported on $D$).
 Since  $n\ge 2$ we have $\dim D>0$ and hence the cohomology group $H^1(\PP^n;R^1\pi_{*}\End_0(\cE))$ might very well be non zero; if that is the case then  $H^2(X,End_0(\cE))\not=0$ by the spectral sequence $H^p(\PP^n;R^q\pi_{*}\End_0(\cE))\Rightarrow H^{p+q}(X,End_0(\cE))$. Notice that if $n=1$ then a non vanishing $R^1\pi_{*}\End_0(\cE)$ contributes to $H^1(X,End_0(\cE))$, not $H^2(X,End_0(\cE))$. 
 
To sum up: Item~(1) suggests that we choose $\cE$ so that it restricts to  a simple sheaf on $X_t$ for $t$ outside a set of codimension at least $2$ in $\PP^n$. Thus if we start with  a random $\cE$ we should perform semistable reduction, as in Subsection~\ref{subsec:innocenzox}. After a finite number of modifications we will get a sheaf $\cE$ restricting to a stable (hence simple) vector bundle on $X_t$ for $t\in\PP^n$ outside a codimension $2$ set. However we do not have  control of the end product, in particular it  is not clear  whether it is modular. 
\subsection{Basic modular sheaves on the Hilbert square of a $K3$ surface}\label{subsec:dibase}
\setcounter{equation}{0}
We describe the first step towards an  analogue  of Subsection~\ref{subsec:innocenzox}. 
Let  $S$ be a smooth projective surface. Let 
$$\tau\colon X_n(S)\to S^n$$
 be the blow up of the big diagonal. The complement of the big diagonal in $S^n$ is identified with a dense open subset $U_n(S)\subset X_n(S)$.  By Proposition 3.4.2 in~\cite{haiman} the natural map $U_n(S)\to S^{[n]}$ extends to a regular map $p\colon X_n(S)\to S^{[n]}$. 
   We let $q_i\colon X_n(S)\to S$ be the composition of $\tau$ and the $i$-th projection $S^n\to S$. Given a locally free sheaf $\cF$ on $S$, let
\begin{equation*}
X_n(\cF):=q_1^{*}(\cF)\otimes\ldots\otimes q_n^{*}(\cF).
\end{equation*}
The action of the symmetric group $\cS_n$ on $S^n$ by permutation of the factors   lifts to  an action  $\rho_n\colon \cS_n\to\Aut(X_n(S))$.  The latter action lifts to a natural action $\rho_n^{+}$ on  $X_n(\cF)$. There is also a twisted  action $\rho_n^{-}=\rho_n^{+}\cdot\chi$ where  $\chi\colon\cS_n\to\{\pm 1\}$ is the sign character. Moreover  $\rho_n^{\pm}$  descends to an action  
$\bm{\rho}_n^{\pm}\colon \cS_n \to \Aut(p_{*}X_n(\cF))$ because it  maps to itself any fiber of  $p\colon X_n(S)\to S^{[n]}$. 
\begin{dfn}
Let $\cF^{\pm}[n]\subset p_{*}X_n(\cF)$ be the sheaf of $\cS_n$-invariants for $\bm{\rho}_n^{\pm}$.
\end{dfn}
The sheaf $\cF^{\pm}[n]$ is reflexive for any $n$, and is locally free for $n\le 2$. Of course $\cF^{\pm}[1]=\cF$ and $\cF^{\pm}[-1]=0$, hence the construction is interesting if $n\ge 2$. 
\begin{prp}\label{yaufever}
Let $S$ be a projective  $K3$ surface, and let $\cF$ be a locally free sheaf on $S$ such that $\chi(S,End\cF)=2$. 
Then $\cF[2]^{\pm}$ is a locally free modular sheaf of rank $r(\cF)^2$, with
\begin{eqnarray}
\Delta(\cF[2]^{\pm}) & = & \frac{r(\cF[2]^{\pm})(r(\cF[2]^{\pm})-1)}{12}c_2(S^{[2]}),  \label{disceffe} \\
d(\cF[2]^{\pm}) & = & 5\cdot {r(\cF[2]^{\pm})\choose 2}, \label{didieffe} \\
a(\cF[2]^{\pm}) & = & \frac{5}{8}r(\cF)^6(r(\cF)^2-1). \label{adieffebis}
\end{eqnarray}
\end{prp}
Proposition~\ref{yaufever} is proved via an explicit computation, for details see Proposition~5.2 in~\cite{ogmodvb}. 
\begin{rmk}
In a recent preprint~\cite{markmod} E.~Markman proved that $\cF[n]^{\pm}$ is modular for all $n$. The same paper contains other very interesting constructions of modular sheaves on HK varieties.
\end{rmk}
We recall that a  locally free sheaf $\cF$ on $S$  is spherical 
if $h^p(S,End^0(\cF))=0$ for all $p$. Notice that if $\cF$ is spherical then $\chi(S,End\cF)=2$ because $End\cF=End^0(\cF)\oplus \cO_S$. 
\begin{prp}\label{prp:seirigido}
Let $S$ be a projective  $K3$ surface. Let $\cF$ be a locally free sheaf on $S$ which is spherical, 
i.e.~such that $h^p(S,End^0(\cF))=0$ for all $p$,  where $End^0(\cF)\subset End(\cF)$ is the  subsheaf of traceless endomorphisms. Then for all $p$ we have
\begin{equation}
h^p(S^{[2]},End^0(\cF[2]^{\pm})) =0.
\end{equation}
\end{prp}
Proposition\ref{prp:seirigido} follows from the McKay correspondence proved by Haiman and Bridgeland-King-Reid: in fact the McKay correspondence gives that
\begin{equation}\label{unduetre}
\Ext^{*}(\cF^{\pm},\cF^{\pm})\cong \Ext^{*}(\cF^{\boxtimes 2},\cF^{\boxtimes 2})^{\ZZ/(2)}\cong
\Sym^2\left(\Ext^{*}(\cF,\cF)\right).
\end{equation}
For the details  see Proposition~5.4 in~\cite{ogmodvb}. 
Below is a remarkable consequence of Proposition~\ref{prp:seirigido} 
(one applies the main result of~\cite{man-iac-pairs}).
\begin{crl}\label{iacman}
Keep hypotheses as in Proposition~\ref{acca20}. Then the natural map between deformation spaces  $\Def(S^{[2]},\cF[2]^{\pm})\lra \Def(S^{[2]},\det\cF[2]^{\pm})$ is smooth.
\end{crl}
\subsection{Existence of stable modular vector bundles  on Lagrangian HK's}\label{subsec:fossanova}
\setcounter{equation}{0}
We describe the second step towards an analogue  of Subsection~\ref{subsec:innocenzox}. The starting point is  an elliptic $K3$ surface $S\to\PP^1$.
\begin{dfn}\label{laghilb}
If $S\to\PP^1$ is an elliptic $K3$ surface,  the \emph{associated Lagrangian fibration} is the   composition 
\begin{equation}
S^{[2]}\to S^{(2)}\to (\PP^1)^{(2)}\cong \PP^2.
\end{equation}
\end{dfn}
We consider  vector bundles $\cF^{\pm }[2]$ on $S^{[2]}$  associated to a vector bundle $\cF$ on $S$ which restricts to a stable vector bundle on every elliptic fiber. To be precise we choose the elliptic $K3$ as in the following claim (which follows from surjectivity of the period map for $K3$ surfaces). 
\begin{clm}\label{eccoell}
Let $m_0,d_0$ be positive natural numbers. There exist $K3$ surfaces $S$  with an elliptic fibration $S\to \PP^1$ such that 
\begin{equation}\label{neronsevero}
H^{1,1}_{\ZZ}(S)=\ZZ[D]\oplus\ZZ[C], \quad
D\cdot D=2m_0,\quad D\cdot C=d_0.
\end{equation}
\end{clm}
Notice that every elliptic fiber on a $K3$ surface as above is irreducible and therefore  slope-stability of a sheaf on a fiber is well defined, i.e.~independent of the choice of a polarization.

The vector bundle $\cF$ on $S$ is chosen to be as in the proposition below.
\begin{prp}\label{rigsuk}
Let $m_0,r_0,s_0$ be positive integers such that  
 $m_0+1=r_0s_0$. Suppose that $d_0$ is an integer coprime to $r_0$, and that
 \begin{equation}\label{digrande}
d_0>\frac{(2m_0+1)r_0^2(r_0^2-1)}{4}.
\end{equation}
Let $S$ be an elliptic $K3$ surface as in Claim~\ref{eccoell}.
 Then there exists a vector bundle $\cF$ on $S$ such that the following hold:
 \begin{enumerate}
\item
 $v(\cF)=(r_0,D,s_0)$,
\item
$\chi(End\cF)=2$,
\item
$\cF$  is $h$ slope-stable for any polarization $h$ of $S$, 
\item
and the restriction of $\cF$ to every elliptic fiber  is slope-stable. 
\end{enumerate}
\end{prp}
\begin{proof}
Let $\vv$ be the Mukai vector $\vv=(r_0,D,s_0)$. Then $\vv^2=-2$ because $m_0+1=r_0s_0$.  By Lemma~\ref{nocamere} and~\eqref{digrande}   there is no $a(\vv)$-wall, and hence all polarizations of $S$ are $a(\vv)$-suitable. 
Since $d_0$ is  coprime to $r_0$ the moduli space $\cM_{\vv}(S,h)$ is a singleton $\{[\cF]\}$  for every  polarization $h$ of $S$. Items~(1)-(3) follow at once. Moreover by 
Subsections~\ref{subsec:torcere} and~\ref{subsec:renatozero} the restriction of $\cF$ to every elliptic fiber  is slope-stable, i.e.~Item~(4) holds. 
\end{proof}
Let $S$ and $\cF$ be as above, and let $\pi\colon S^{[2]}\to\PP^2$ be the associated Lagrangian fibration. Recall that $\PP^2$ parametrizes elements of the symmetric square $(\PP^1)^{(2)}$. Let $x_1+x_2$ be such an element, where $x_1\not= x_2\in\PP^1$. Then $\pi^{-1}(x_1+x_2)$ is isomorphic to $C_{x_1}\times C_{x_2}$, and
\begin{equation}\label{biancolatte}
\cF[2]^{\pm}_{|\pi^{-1}(x_1+x_2)}\cong \cF_{x_1}\boxtimes \cF_{x_2}. 
\end{equation}
Since $\cF_{x_1}, \cF_{x_2}$ are slope stable, it follows that the above vector bundle is slope stable. 
Hence the modular vector bundle $\cF[2]^{\pm}$ restricts to a slope stable vector bundle on a generic fiber of the Lagrangian fibration 
$\pi\colon S^{[2]}\to\PP^2$. Since $a(\cF)$-suitable polarizations exist, this gives a large set of examples of slope stable modular vector bundles on $S^{[2]}$.  The following is a key result.
\begin{prp}\label{acca20}
Let $S$ be a $K3$ surface with an elliptic fibration $S\to \PP^1$ as in Claim~\ref{eccoell}, and  let $\cF$ be a vector bundle on $S$ as in Proposition~\ref{rigsuk}. Then
the restriction of $\cF[2]^{\pm}$ to every fiber of the associated Lagrangian fibration $\pi\colon S^{[2]}\to\PP^2$ is simple.
\end{prp}
For the proof see Proposition~6.7 in~\cite{ogmodvb}. 
\subsection{Good stable vector bundles on a generic $[(X,h)]\in \cN_e^i(d)$}\label{subsec:vblag}
\setcounter{equation}{0}
Below is an analogue  of the main result of Subsection~\ref{subsec:innocenzox}. 
\begin{prp}\label{buonacompt}
Let $i\in\{1,2\}$ and let $r_0>0$ be such that $i\equiv r_0\pmod{2}$. Suppose  that~\eqref{econ}  holds, that $e\notdivides 2d$ and that 
\begin{equation}\label{pazienza}
d>\frac{5}{16}r_0^6(r_0^2-1)(e+1).
\end{equation}
Then there exists an irreducible component $ \cN_e^i(d)^{\rm good}$ of $ \cN_e^i(d)$ such that the following holds. 
Let $[(X,H)]\in \cN_e^i(d)^{\rm good}$ be  generic, and hence Proposition~\ref{unicafibr} gives a well defined  associated Lagrangian fibration 
$\pi\colon X\to\PP^2$.
Then there exists  an $h$ slope-stable vector bundle 
$\cE$ on $X$ such that~\eqref{ele} holds, i.e.
 \begin{equation}\label{hubris}
r(\cE)=r_0^2,\quad c_1(\cE)=\frac{r_0}{i} h,\quad \Delta(\cE)  =  \frac{r(\cE)(r(\cE)-1)}{12}c_2(X),
\end{equation}
and  the pull-back of $\cE_t$ to  the normalization of  $X_t$  
  is    slope-stable (for the pull back of $h_{|X_t}$)  except possibly for a finite set of $t\in\PP^2$. 
\end{prp}
We sketch how one gets Proposition~\ref{buonacompt} from the results in the previous subsections and we refer to Proposition~7.2 in~\cite{ogmodvb} for  a complete proof. 

Since Proposition~\ref{buonacompt} is trivially true if $r_0=1$ we may suppose that $r_0\ge 2$.
Our first observation is that one can choose a vector bundle $\cF$ as in Proposition~\ref{rigsuk} so that~\eqref{hubris} holds for  $\cE:=\cF^{\pm}[2]$. 
The proof of the result below is elementary.
\begin{lmm}\label{champollion}
Let $i\in\{1,2\}$. Let  $e,r_0$ be positive natural numbers such  that $r_0\equiv i\pmod{2}$ and~\eqref{econ} holds. Let 
\begin{equation}\label{emmezero}
m_0:=
\begin{cases}
\frac{e}{2}+\frac{(r_0-1)^2}{4} & \text{if $r_0$ is odd,} \\
\frac{e}{8}+\frac{(r_0-1)^2}{4} & \text{if $r_0$ is even.} 
\end{cases}
\end{equation}
($m_0$ is an integer by~\eqref{econ}.) There exists an integer $s_0$ such that $m_0+1=r_0 s_0$. 
\end{lmm}
Let $i,e,r_0,m_0$ be as in Lemma~\ref{champollion}.
Suppose that $d_0$ is an integer  coprime to $r_0$  such that~\eqref{digrande} holds. Let $S$ be an elliptic $K3$ surface as in Claim~\ref{eccoell}, and let $\cF$ be a vector bundle  on $S$ as in Proposition~\ref{rigsuk}. Let
\begin{equation}\label{accapiumeno}
h^{\pm}:=\mu(c_1(\cF))-\frac{r_0\mp 1}{2}\delta
\end{equation}
where $\mu\colon H^2(S)\to H^2(S^{[2]})$ is the map in~\eqref{mappamu}.  Lastly let
\begin{equation}\label{poler}
h:=ih^{+}.
\end{equation}
Straighftorward computations give the result below.
\begin{prp}\label{rosetta}
Let  $\cE:=\cF[2]^{+}$. Then the following hold:
\begin{enumerate}
\item
 $h$ is a primitive cohomology class, $q(h)=e$ and
$q(h,H^2(X;\ZZ))=(i)$,
\item
$r(\cE)$, $c_1(\cE)$ and $\Delta(\cE)$ are given by~\eqref{hubris}.
\end{enumerate}
\end{prp}
\begin{rmk}
One gets an analogue of Proposition~\ref{rosetta} with  $\cF[2]^{-}$ replacing $\cF[2]^{+}$ provided
$\frac{(r_0-1)^2}{4}$ is replaced  by $\frac{(r_0+1)^2}{4}$ in~\eqref{emmezero} and $h=ih^{-}$.
\end{rmk}
Let $S$ be an elliptic $K3$ surface as in Claim~\ref{eccoell} and let  $\cF$ be a vector bundle  on $S$ as in Proposition~\ref{rigsuk}. One gets vector bundles satisfying the thesis of Proposition~\ref{buonacompt} by deforming  $S^{[2]}$ and $\cF[2]^{+}$.  

More precisely let  $X(0)=S^{[2]}$  let 
 $h(0):=h$, where $h$ is given by~\eqref{poler}, and
 let $\cE(0):=\cF[2]^{+}$.  Let $C\subset S$ be a fiber of the elliptic fibration and let 
$f(0):=\mu(\cl(C))$. Lastly let $d_0$ be as in~\eqref{neronsevero} and set
\begin{equation}
d:=i d_0.
\end{equation}
Then the sublattice $\la f(0),h(0)\ra\subset H^{1,1}_{\ZZ}(X(0))$ is saturated and
\begin{equation}\label{edizero}
q(f(0),f(0))=0,\quad q(h(0),f(0))=d,\quad q(h(0),h(0))=e.
\end{equation}
Let $\pi(0)\colon X(0)\to\PP^2$ be the Lagrangian fibration associated to the elliptic fibration of $S$. 

Let $\varphi\colon\cX\to B$ be an analytic representative of the deformations space of $(X(0),\la h(0),f(0)\ra)$ i.e.~deformations of $X(0)$ that keep $h(0)$ and $f(0)$ of Hodge type. 
We assume that $B$ is contractible. Let  $0\in B$ the base point, in particular $X(0)$ is isomorphic to $\varphi^{-1}(0)$. For $b\in B$ we let $X(b):=\varphi^{-1}(b)$. 
If $B$ is small enough, then by  Proposition~\ref{acca20} and Corollary~\ref{iacman} the vector bundle $\cE(0)$ on $X(0)$ deforms to a vector bundle $\cE(b)$ on $X(b)$ (unique up to isomorphism because $H^1(X(0),End_0\cE(0))=0$). Notice that $\la h(0),f(0)\ra$ deforms  by Gauss-Manin parallel transport to a saturated sublattice 
\begin{equation}\label{lambdabi}
\la h(b),f(b)\ra\subset H^{1,1}_{\ZZ}(X(b)).
\end{equation}
Possibly after shrinking $B$ around $0$ there exists a map $\pi\colon\cX\to \PP^2$ which  restricts to  a Lagrangian fibration $X(b)\to\PP^2$ for
 every $b\in B$, and is equal to  $\pi(0)$  on $X(0)$. If $\pi(b)$ is the restriction of $\pi$ to $X(b)$
then $f(b)=c_1(\pi(b)^{*}\cO_{\PP^2}(1))$.
\begin{prp}\label{chiave}
With  hypotheses and notation as above, the following holds.  For $b\in B$ outside a proper analytic subset  
$h(b)$ is ample and  
\begin{equation}\label{behemoth}
[(X(b),h(b))]\in\cN^i_e(d),\quad  i\equiv r_0\pmod{2}.
\end{equation}
Moreover $\cE(b)$ is $h(b)$ slope-stable, \eqref{hubris} holds  
and   the pull-back of $\cE(b)_t$ to  the normalization of  $X(b)_t$  
  is    slope-stable (for the pull back of $h(b)_{|X(b)_t}$)  except possibly for a finite set of $t\in\PP^2$.
\end{prp}
\begin{proof}[Idea of Proof]
For a very general $b\in B$ we have
\begin{equation}\label{tempiduri}
\la h(b),f(b)\ra= H^{1,1}_{\ZZ}(X(b)).
\end{equation}
Since the set of $b\in B$ for which the thesis of the proposition holds is Zariski open, it suffices to   prove that if~\eqref{tempiduri} holds then $h(b)$ is ample, Equation~\eqref{behemoth} holds,  $\cE(b)$ is $h(b)$ slope-stable etc. Using Lemma~\ref{nocamere} one shows that $h(b)$ is ample for such $b$ (a numerical characterization of ample divisors on HK's of Type $K3^{[n]}$ is known). Equation~\eqref{behemoth} holds by Item~(1) of Proposition~\ref{rosetta}. Using again Lemma~\ref{nocamere} one shows  that there is a single open $a(\cE(b))$-chamber i.e.~the ample cone. In particular $h(b)$ is 
$a(\cE(b))$-suitable and thus $\cE(b)$ is slope stable because it restricts to a slope stable vector bundle on a generic Lagrangian fiber, see~\eqref{biancolatte}. Equation \eqref{hubris} holds by Item~(2) of Proposition~\ref{rosetta}. 
More than one argument enters into the proof of  the last statement, one of them is Item~(2)  Proposition~\ref{acca20}.  
\end{proof}

\begin{proof}[Sketch of proof of Proposition~\ref{buonacompt}]
Let $\varphi\colon\cX\to B$ be as above. Let  $B^{\rm good}\subset B$ be a non empty Zariski open such that the thesis
of Proposition~\ref{chiave}  holds for all $b\in B^{\rm good}$. Then $B^{\rm good}$ parametrizes polarized varieties $(X(b),h(b))$ for which the moduli point belongs to $\cN_e^i(d)$. Since  the moduli map $m\colon B^{\rm good}\to \cN_e^i(d)$ has finite fibers and $\dim B=\dim\cN_e^i(d)$, the Zariski closure of $m(B^{\rm good})$ is an irreducible component of $\cN_e^i(d)$ that we name  $\cN_e^i(d)^{\rm good}$. Then Proposition~\ref{buonacompt} holds by Proposition~\ref{chiave}.
\end{proof}

\subsection{Proof of Theorem~\ref{unicita}}
\setcounter{equation}{0}
\begin{prp}\label{unico}
Let $i\in\{1,2\}$. Suppose  that $r_0\equiv i\pmod{2}$,  that~\eqref{econ}  holds, that $e\notdivides 2d$ and that 
\begin{equation}\label{santapi}
d>\frac{5}{16}r_0^6(r_0^2-1)(e+1).
\end{equation}
Let $[(X,H)]\in \cN_e^i(d)^{\rm good}$ be a generic point, where $\cN_e^i(d)^{\rm good}$  is as in Proposition~\ref{buonacompt}. Then, up to isomorphism,  there exists  one and only one $h$ slope-stable vector bundle 
$\cE$ on $X$ such that~\eqref{ele} holds. 
\end{prp}
Of course there exists at least  one $h$ slope-stable vector bundle 
$\cE$ on $X$ such that~\eqref{ele} holds by Proposition~\ref{buonacompt}, the new result is that it is unique up to isomorphism. 
The proof of Proposition~\ref{unico} is obtained by adapting the unicity statement in the proof of  Proposition~\ref{enerzero}. A key r\^ole is played by Proposition~\ref{propriostab} and the last sentence of Proposition~\ref{buonacompt}.
For details see Proposition~7.5 in~\cite{ogmodvb}. 

Let us prove Theorem~\ref{unicita}. If $r_0=1$ the result is trivially true, hence we may assume that $r_0\ge 2$. Let $\cX\to T_{e}^1$ and $\cX\to T_{e}^2$ be  complete families of  polarized 
 HK's of Type $K3^{[2]}$ such that~\eqref{divuno}, respectively ~\eqref{divdue}, holds - e.g.~the families parametrized by the relevant  open subsets of  suitable Hilbert schemes. Since $\cK^i_e$ is irreducible we may, and will, assume that $T_{e}^i$ is irreducible.  By passing to normalization if necessary we may  assume that $T_{e}^i$ is normal. For $t\in T_{e}^i$ we let $(X(t),h(t))$ be the corresponding polarized HK of Type $K3^{[2]}$. We let $m\colon T_e^i\to\cK^i_e$ be the moduli map, sending $t$ to $[(X(t),h(t))]$.
 
By fundamental results of Gieseker and Maruyama  there exists a  map of schemes
\begin{equation}\label{vbrel}
f\colon \cM_{e}(r_0)\to T^i_e 
\end{equation}
 such that for every $t\in T^i_e$ the (scheme theoretic) fiber $f^{-1}(t)$  is isomorphic to the (coarse) moduli space  of 
  $h(t)$ slope-stable vector bundles $\cE$ on $X(t)$ such that~\eqref{ele} holds. Moreover  $f\colon \cM_{e}(r_0)\to T^i_e$ is  of finite type by Maruyama~\cite{marubound}, and hence  $f(\cM_{e}(r_0))$ is a  constructible subset of $T^i_e$.

 By Proposition~\ref{unico} for $t$ in   a dense subset of  
 \begin{equation}\label{manydiv}
\bigcup\limits_{d\gg 0}m^{-1}(\cN_e^i(d)^{\rm good}) 
\end{equation}
 the preimage $f^{-1}(t)$ is a singleton. The set in~\eqref{manydiv} is a union of pairwise distinct divisors and hence 
 is Zariski dense in $T^i_e$. Since
$f(\cM_{e}(r_0))$ is a  constructible subset of $T^i_e$, it follows that for  generic $t\in  T^i_e$ the fiber 
$f^{-1}(t)$ is a singleton.

Let $[\cE]$ be the unique point of $f^{-1}(t)$ for $t$ a generic point of $m^{-1}(\cN_e^i(d)^{\rm good})$, where $d\gg 0$. Then $H^p(X(t), End_0\cE)=0$ by  Proposition~\ref{acca20}. Hence the last sentence of Theorem~\ref{unicita} follows from upper semicontinuity of cohomology.

 \bibliography{ref-survey-on-modular-sheaves}
 \end{document}